\documentclass[12pt]{amsart}
\usepackage{amssymb}
\usepackage{amsfonts}
\usepackage{latexsym}
\usepackage{amscd}
\usepackage[mathscr]{euscript}
\usepackage{xy} \xyoption{all}
\newcommand{\Z}{{\mathbb{Z}}}

\newcommand{\C}{{\mathbb{C}}}
\newcommand{\T}{{\mathbb{T}}}

\newcommand{\calA}{{\mathcal A}}

\newcommand{\calC}{{\mathcal C}}

\newcommand{\calH}{{\mathcal H}}
\newcommand{\calL}{{\mathcal L}}
\newcommand{\calO}{{\mathcal O}}

\newcommand{\calV}{{\mathcal V}}
\newcommand{\calX}{{\mathcal X}}

\newcommand{\ol}{\overline}
\newcommand{\uloopr}[1]{\ar@'{@+{[0,0]+(-4,5)}@+{[0,0]+(0,10)}@+{[0,0] +(4,5)}}^{#1}}
\newcommand{\uloopd}[1]{\ar@'{@+{[0,0]+(5,4)}@+{[0,0]+(10,0)}@+{[0,0]+ (5,-4)}}^{#1}}
\newcommand{\dloopr}[1]{\ar@'{@+{[0,0]+(-4,-5)}@+{[0,0]+(0,-10)}@+{[0, 0]+(4,-5)}}_{#1}}
\newcommand{\dloopd}[1]{\ar@'{@+{[0,0]+(-5,4)}@+{[0,0]+(-10,0)}@+{[0,0 ]+(-5,-4)}}_{#1}}

\newcommand{\luloop}[1]{\ar@'{@+{[0,0]+(-8,2)}@+{[0,0]+(-10,10)}@+{[0, 0]+(2,2)}}^{#1}}

\newcommand{\FSGr}{\mathbf{FSGr}}

\newcommand{\Atil}{\tilde{A}}

\DeclareMathOperator{\coker}{coker}

\newcommand{\mon}[1]{\calV(#1)} %Monoide de projectius

\newcommand{\Si}{{\rm Sink}}
 \DeclareMathOperator{\Tr}{Tr}

\newcommand{\Cstaralg}{C^*\text{-}\mathbf{alg}}
\DeclareMathOperator{\rd}{red}
\newcommand{\Cstred}{C^*_{\rd}}
\newcommand{\bgast}{\mbox{\Large$*$}}

\DeclareMathOperator{\diag}{diag}
\DeclareMathOperator{\trace}{trace}

\theoremstyle{plain}
\newtheorem{theorem}{Theorem}[section]
\newtheorem{lemma}[theorem]{Lemma}
\newtheorem{proposition}[theorem]{Proposition}
\newtheorem{corollary}[theorem]{Corollary}

\theoremstyle{definition}
\newtheorem{definition}[theorem]{Definition}
\newtheorem{example}[theorem]{Example}

\newtheorem{remark}[theorem]{Remark}

\newtheorem*{remark*}{Remark}
\newtheorem*{remarks*}{Remarks}

\newtheorem*{assumption*}{Assumption}

\newtheorem{openproblem}[theorem]{Problem}
\newtheorem{subsec}[theorem]{}

\numberwithin{equation}{section}

\begin{document}

\title[C*-algebras of separated graphs]{C*-algebras of separated graphs}%

\author{P. Ara}

\address{Departament de Matem\`atiques, Universitat Aut\`onoma de Barcelona,
08193 Bellaterra (Barcelona), Spain.} \email{para@mat.uab.cat}

\author{K. R. Goodearl}

\address{Department of Mathematics, University of
California, Santa Barbara, CA 93106. }\email{goodearl@math.ucsb.edu}
%\dedicatory{}

\thanks{The first-named author was partially supported by DGI MICIIN-FEDER
MTM2008-06201-C02-01, and by the Comissionat per Universitats i
Recerca de la Generalitat de Catalunya.}



\date{\today}

\begin{abstract}
The construction of the C*-algebra associated to a directed graph
$E$ is extended to incorporate a family $C$ consisting of partitions
of the sets of edges emanating from the vertices of $E$. These
C*-algebras $C^*(E,C)$ are analyzed in terms of their ideal theory
and K-theory, mainly in the case of partitions by finite sets. The
groups $K_0(C^*(E,C))$ and $K_1(C^*(E,C))$ are completely described
via a map built from an adjacency matrix associated to $(E,C)$. One
application determines the K-theory of the C*-algebras
$U^{\text{nc}}_{m,n}$, confirming a conjecture of McClanahan.  A
reduced C*-algebra $\Cstred(E,C)$ is also introduced and studied. A
key tool in its construction is the existence of canonical faithful
conditional expectations from the C*-algebra of any row-finite graph
to the C*-subalgebra generated by its vertices. Differences between
$\Cstred(E,C)$ and $C^*(E,C)$, such as simplicity versus
non-simplicity, are exhibited in various examples, related to some
algebras studied by McClanahan.
\end{abstract}
\maketitle

%%%%%%%%%%%%%%%%%%%%%%%%%%%%%%%%%%%%%
\section{Introduction}
\label{sect:intro}

Graph C*-algebras constitute an important class of
C*-algebras, providing models for the classification theory
and a rich source of examples and inspiration. Among the most
basic examples of graph C*-algebras are the Cuntz algebras
$\mathcal O _n$, initially studied by Cuntz \cite{Cuntz77, Cuntz81},
and the Cuntz-Krieger algebras  \cite{CuntzKrieger} associated to
finite square matrices with entries in $\{0,1\}$. We refer the
reader to \cite{Raeburn} for further information on this important
class of C*-algebras.
\smallskip

The present paper addresses the structure of a new class of graph
C*-algebras, associated to \emph{separated graphs} $(E,C)$, where
$E$ is a directed graph and $C$ is a family that gives a partition
of the set of edges departing from each vertex of $E$. (These
algebras have also been recently introduced by Duncan \cite{Dun}, with different notation, as
\emph{C*-algebras of edge-labelled graphs}, which should not be
confused with the \emph{labelled graph C*-algebras} developed by
Bates and Pask \cite{BP}. Our viewpoint, which was developed in \cite{AG} for the algebraic case, appears to be more flexible and better adapted to the construction and analysis of these algebras.) It was shown in \cite{AG} how to
associate to any such separated graph $(E,C)$ a complex *-algebra
$L(E,C)$, called the \emph{Leavitt path algebra} of the separated
graph $(E,C)$. We may define the C*-algebra $C^*(E,C)$ of the
separated graph $(E,C)$ as the universal C*-envelope of $L(E,C)$. We also introduce a reduced version, denoted $\Cstred(E,C)$, in the
case that $(E,C)$ is \emph{finitely separated}, meaning that the
partitions in $C$ consist of finite sets. To glimpse the differences
and similarities between the full and reduced graph C*-algebras, let
us mention the following facts. When we consider a separated graph
$(E,C)$ with just one vertex and the sets in the partition $C$ are
reduced to singletons, the full graph C*-algebra $C^*(E,C)$ is just
the full group C*-algebra $C^*(\mathbb F )$ of a free group $\mathbb
F$ of rank $|E^1|$, while the reduced graph C*-algebra is precisely
the reduced group C*-algebra $C^*_r (\mathbb F)$. On the other hand,
when we deal with a {\it trivially separated graph} $(E,C)$ (meaning
that for each non-sink $v\in E^0$, the partition $C_v$ consists of
the single set $s^{-1}(v)$), then both the full graph C*-algebra
$C^*(E,C)$ and the reduced graph C*-algebra $\Cstred(E,C)$ coincide
with the usual graph C*-algebra $C^*(E)$ (Theorem
\ref{thm:red-prod-inj}(2)). In general, the behaviors of the full and
reduced graph C*-algebras are quite different, as suggested by the
free group C*-algebra example above. We consider specific examples
in Section \ref{sect:idealsinCred}, for which we show that the
reduced graph C*-algebra is simple, including in particular algebras
closely related to the C*-algebras considered by Brown and
McClanahan, see \cite{Brown, McCla1, McCla2, McCla3}. Indeed, as we
show in Section \ref{sect:relMcClan}, our examples (both reduced and
full) are Morita-equivalent to ones considered in the abovementioned
papers.

We also compute, using a result of Thomsen, the $K$-theory of the
full graph C*-algebras of finitely separated graphs $(E,C)$, obtaining a formula
that very much resembles the one known for ordinary graph
C*-algebras, as stated for instance in \cite[Theorem  3.2]{RS}. Namely, $K_0(C^*(E,C))$ and $K_1(C^*(E,C))$ are the cokernel and kernel of a map between free abelian groups given by an identity minus an adjacency matrix associated to $(E,C)$ (see Theorem  \ref{thm:KTHsepgraph}).

An important ingredient in our work is the construction of a
canonical {\it faithful} conditional expectation $C^*(E)\to
C_0(E^0)$ for any row-finite graph $E$ (see Section
\ref{sect:cond-expecs}).

\begin{subsec}
{\bf Contents.} We now explain in more detail the contents of this
paper. The definitions of a separated graph $(E,C)$ and its Leavitt
path algebra $L(E,C)$ and full C*-algebra $C^*(E,C)$ are given in
Subsection \ref{background}. We construct  canonical faithful
conditional expectations $\Phi _E\colon C^*(E)\to C_0(E^0)$ for all row-finite graphs $E$ in
Section \ref{sect:cond-expecs}. The reduced graph C*-algebras
$\Cstred(E,C)$ are introduced in Section \ref{Cstar}, based on the conditional expectations constructed in the previous section. Here we make
use of the theory of full and reduced amalgamated free products of
C*-algebras (\cite{voicu, VDN}). We show that the Leavitt path
algebra $L(E,C)$ embeds in the reduced graph C*-algebra $\Cstred
(E,C)$ (and thus also embeds in the full graph C*-algebra
$C^*(E,C)$), and that, for a trivially separated row-finite graph
$E$, we have $C^*(E)\cong \Cstred (E)$ canonically (Theorem
\ref{thm:red-prod-inj}). We also exhibit a family of closed ideals
of $C^*(E,C)$, parametrized by the lattice $\mathcal H$ of
hereditary $C$-saturated subsets of $E^0$ (Corollary
\ref{Mlattinclusion}). We show simplicity of the reduced graph
C*-algebras $\Cstred (E,C)$ for various families of finitely
separated graphs in Section \ref{sect:idealsinCred}, including the
separated graphs giving rise to C*-algebras analogous to the ones
considered by Brown and McClanahan in \cite{Brown, McCla1, McCla2,
McCla3}. We also show in Proposition \ref{prop:heredbreak} that
there are examples of finitely separated graphs $(E,C)$ for which
$\Cstred (E,C)$ is simple but the lattice of hereditary
$C$-saturated subsets of $E^0$ has more than two elements, so
$C^*(E,C)$ is not simple. This example also shows that the structure
of projections in the full and reduced graph C*-algebras can be
quite different. Section \ref{sect:K-theory} is devoted to the
computation of $K$-theory of full graph C*-algebras. We obtain a
quite satisfying formula in Theorem \ref{thm:KTHsepgraph}, using a
powerful result of Thomsen \cite[Theorem 2.7]{Thomsen}. This in
particular enables us to confirm a conjecture of McClanahan on the
$K$-theory of the C*-algebras $U^{\text{nc}}_{m,n}$. The exact
relationship of the reduced graph C*-algebras $\Cstred
(E(m,n),C(m,n))$ and the examples considered in \cite{McCla3} is
established in Section \ref{sect:relMcClan}. (See Example
\ref{ex:EmnCmn} for the definition of the separated graph
$(E(m,n),C(m,n))$.) By using this connection and some results in the
literature, we establish that $\Cstred (E(n,n),C(n,n))$, for $n>1$,
is a simple C*-algebra of stable rank one, with a unique tracial
state, and having minimal projections (Corollary \ref{cor:C*Enn}).
We end the paper with a discussion of open problems.
\end{subsec}

\smallskip

\begin{subsec}  \label{background}
{\bf Background definitions.}
Throughout, all graphs will be directed graphs of the form $E= (E^0,E^1,s,r)$, where $E^0$ and $E^1$ denote the sets of vertices and edges of $E$, respectively, and $s,r: E^1 \rightarrow E^0$ are the source and range maps. No cardinality restrictions are imposed on $E^0$ and $E^1$. We follow the convention of composing paths from left to right -- thus, a path in $E$ is given in the form $\alpha= e_1e_2 \cdots e_n$ where the $e_i\in E^1$ and $r(e_i)= s(e_{i+1})$ for $i<n$. The \emph{length} of such a path is $|\alpha| := n$. Paths of length $0$ are identified with the vertices of $E$.

\begin{definition} \label{defsepgraph} \cite[Definition 2.1]{AG}
A \emph{separated graph} is a pair $(E,C)$ where $E$ is a graph,  $C=\bigsqcup
_{v\in E^ 0} C_v$, and
$C_v$ is a partition of $s^{-1}(v)$ (into pairwise disjoint nonempty
subsets) for every vertex $v$. (In case $v$ is a sink, we take $C_v$
to be the empty family of subsets of $s^{-1}(v)$.)

If all the sets in $C$ are finite, we say that $(E,C)$ is a \emph{finitely separated} graph. This necessarily holds if $E$ is row-finite.

The set $C$ is a \emph{trivial separation} of $E$ in case $C_v= \{s^{-1}(v)\}$ for each $v\in E^0\setminus \Si(E)$. In that case, $(E,C)$ is called a \emph{trivially separated graph} or a  \emph{non-separated graph}. Any graph $E$ may be paired with a trivial separation and thus viewed as a trivially separated graph.
\end{definition}

The concept of a separated graph is related to that of an \emph{edge-colored graph}, that is, a pair $(E,f)$ where $E$ is a (directed) graph and $f:E^1 \to N$ is a function from $E^1$ to some set $N$. Given such a pair, set
$$C_v := \{ s^{-1}(v)\cap f^{-1}(n) \mid n\in N\text{\;and\;} s^{-1}(v)\cap f^{-1}(n) \ne\emptyset \}$$
for $v\in E^0$ and $C=\bigsqcup
_{v\in E^ 0} C_v$. Then $(E,C)$ is a separated graph. Conversely, given a separated graph $(E,C)$, the map $f: E^1 \rightarrow C$ such that $e\in f(e)$ for $e\in E^1$ is an edge-coloring of $E$. The general definition of an edge-coloring allows edges with different sources to receive the same color. However, no relations between such edges are imposed in the C*-algebras we construct.

\begin{definition} \label{def:LPASG} \cite[Definition 2.2]{AG}
For any separated graph $(E,C)$, the (\emph{complex}) \emph{Leavitt path algebra of
$(E,C)$} is the complex *-algebra $L(E,C)$ with generators $\{ v,
e\mid v\in E^0,\ e\in E^1 \}$, subject to the following relations:
\begin{enumerate}
\item[] (V)\ \ $vw = \delta_{v,w}v$ \ and $v=v^*$ \ for all $v,w \in E^0$ ,
\item[] (E)\ \ $s(e)e=er(e)=e$ \ for all $e\in E^1$ ,
\item[] (SCK1)\ \ $e^*f=\delta _{e,f}r(e)$ \ for all $e,f\in X$, $X\in C$, and
\item[] (SCK2)\ \ $v=\sum _{ e\in X }ee^*$ \ for every finite set $X\in C_v$, $v\in E^0$.
\end{enumerate}

\end{definition}

\begin{definition} \label{defC*sepgraph} The \emph{graph C*-algebra} of a separated graph $(E,C)$ is the C*-algebra $C^*(E,C)$  with generators $\{ v, e \mid v\in E^0,\
e\in E^1 \}$, subject to the relations (V), (E), (SCK1), (SCK2). In
other words, $C^*(E,C)$ is the enveloping C*-algebra of $L(E,C)$. This C*-algebra exists because the generating set consists of partial isometries.

In case $(E,C)$ is trivially separated, $C^*(E,C)$ is just the classical graph C*-algebra $C^*(E)$.

For $v\in E^0$ and $e\in E^1$, we use the same symbols $v$ and $e$
for the canonical images of $v$ and $e$ in $C^*(E,C)$. This allows
us to conveniently abbreviate various expressions -- for instance,
if $H\subseteq C^*(E,C)$, we can write $E^0\cap H$ for the set of
those $v\in E^0$ whose canonical images in $C^*(E,C)$ lie in $H$.

By definition, there is a unique
*-homomorphism  $L(E,C) \rightarrow
C^*(E,C)$ sending the generators of $L(E,C)$ to their canonical
images in $C^*(E,C)$. This
*-homomorphism will be called the \emph{canonical map from $L(E,C)$
to $C^*(E,C)$.}

The $C^*(E,C)$ construction also produces the \emph{C*-algebras of edge-colored graphs} introduced by Duncan \cite[Definition 6]{Dun} (although he only considers edge-colorings with natural number values). Since Duncan allows arrows with different sources to have the same color, his construction can produce the same algebra from many different edge-colorings of a given graph.
\end{definition}

In the present paper, we mostly restrict our attention to finitely separated graphs and their C*-algebras.

The natural category of finitely separated graphs is the category $\FSGr$ defined in \cite[Definition 8.4]{AG}. Its objects are all finitely separated graphs $(E,C)$. A morphism from $(F,D)$ to $(E,C)$ in $\FSGr$ is any graph morphism $\phi:F\rightarrow E$ such that
\begin{enumerate}
\item $\phi^0$ is injective.
\item For each $v\in F^0$ and each $X\in D_v$, there is some $Y\in C_{\phi^0(v)}$ such that $\phi^1$ induces a bijection $X\rightarrow Y$.
\end{enumerate}
Condition (2) does not imply that $\phi^1$ is injective, since it might map two different sets in $D_v$ to the same member of $C_{\phi^0(v)}$.

A \emph{complete subobject} of an object $(E,C)$ in $\FSGr$ is any object $(F,D)$ such that $F$ is a subgraph of $E$ and
\begin{enumerate}
\item[(3)] $D_v= \{Y\in C_v \mid Y\cap F^1 \ne \emptyset\}$ for all $v\in F^0$. (In particular, this requires that each set in $C$ which meets $F^1$ must be contained in $F^1$.)
\end{enumerate}
(This is the specialization of \cite[Definition 3.4]{AG} to
$\FSGr$.) Observe that indeed $(F,D)$ is a complete subobject of
$(E,C)$ if and only if $F$ is a subgraph of $E$ and $D$ is a subset
of $C$. In this case, the
inclusion $F\rightarrow E$ (that is, the pair of inclusions
$(F^0\rightarrow E^0,\, F^1\rightarrow E^1)$) is a morphism in
$\FSGr$.

Any morphism $\phi: (F,D) \rightarrow (E,C)$ in $\FSGr$ induces a
unique C*-algebra homomorphism $C^*(\phi): C^*(F,D) \rightarrow
C^*(E,C)$ sending
\begin{equation}  \label{natmapL}
v\longmapsto \phi ^0(v),\qquad\quad e\longmapsto \phi ^1(e)
 \end{equation}
for $v\in F^0$ and $e\in F^1$, since the elements $\phi^0(v)$,
$\phi^1(e)$ satisfy the defining relations of $C^*(F,D)$. The
assignments $(F,D) \mapsto C^*(F,D)$ and $\phi \mapsto C^*(\phi)$
define a functor $C^*(-)$ from $\FSGr$ to the category $\Cstaralg$
of C*-algebras. The argument of \cite[Proposition 3.6]{AG},
\emph{mutatis mutandis}, yields the following result:

\begin{proposition}  \label{FSGrC*functor}
The functor $C^*(-): \FSGr \rightarrow \Cstaralg$ is continuous.
\qed\end{proposition}
\end{subsec}

\section{The canonical conditional expectation}
\label{sect:cond-expecs}

In this section we define the canonical conditional expectation
$\Phi _E\colon C^*(E)\to C_0(E^0)$ for a row-finite directed graph $E$ and we
show its faithfulness.  We will use  these conditional
expectations (for various subgraphs) to define the reduced graph C*-algebra of a finitely separated
graph (see Section \ref{Cstar}). In the following, we identify the C*-algebra of the edgeless graph $(E^0,\emptyset)$ with the function algebra $C_0(E^0)$ on the discrete set $E^0$. Recall that the canonical *-homomorphism $C^*((E^0,\emptyset)) \rightarrow C^*(E)$ is an embedding (e.g., \cite[Theorem 2.1]{BHRS}). We thus identify $C_0(E^0)$ with the sub-C*-algebra of $C^*(E,C)$ generated by $E^0$.

\begin{theorem}  \label{3condexp}
Let $E$ be a row-finite graph. Then there exists a unique
conditional expectation
$$\Phi_E\colon C^*(E)\longrightarrow C_0(E^0)$$
such that, for all paths $\gamma, \nu$ in $E$, we have
\begin{equation}  \label{basicPhi}
\Phi_E(\gamma\nu^*)= \begin{cases}  0  &(\text{if\ } \gamma\ne\nu)\\
\biggl( \prod_{i=1}^n |s^{-1}s(e_i)| \biggr)^{-1} s(\gamma)
&(\text{if\ } \gamma=\nu= e_1e_2\cdots e_n \text{\ for some\ }
e_i\in E^1). \end{cases}
\end{equation}
Moreover the conditional expectation $\Phi _E$ is faithful.
\end{theorem}

\begin{proof} Uniqueness is clear in case of existence.

Let $E$ be a row-finite graph. The map $\Phi_E$ will be defined as
the composition of three maps: $\Phi _E=\Phi_3\circ \Phi _2\circ
\Phi _1$, each of which is a faithful conditional expectation.
The first of these maps is the canonical conditional expectation
$\Phi _1\colon C^*(E)\to C^*(E)^{\alpha}$, where $\alpha \colon
\T\to \text{Aut}(C^*(E))$ is the gauge action (e.g., \cite[p.~1161]{BHRS}) and $C^*(E)^{\alpha}$
is the fixed point C*-algebra, which is the AF-subalgebra of
$C^*(E)$ generated by all the paths $\alpha \beta^*$, where $\alpha,
\beta $ are finite paths in $E$ such that $r(\alpha)=r(\beta)$ and
$|\alpha|=|\beta |$. The conditional expectation $\Phi _1$ is
faithful by \cite[Proposition 3.2]{Raeburn}.

The second conditional expectation $\Phi _2$ appearing in the
definition of $\Phi _E$ is the unique conditional expectation $\Phi
_2\colon C^*(E)^{\alpha}\to D$ from the AF-algebra $C^*(E)^{\alpha}$
to its canonical Cartan subalgebra $D$, where $D$ is the commutative
diagonal AF-algebra generated by $\lambda \lambda ^*$, $\lambda \in
E^*$. Indeed, since every AF-groupoid is amenable (\cite[Remark
III.1.2]{Renault}), it follows from \cite[Theorem II.4.15]{Renault}
that $D$ is the image of a unique conditional expectation $\Phi
_2\colon C^*(E)^{\alpha}\to D$, which is faithful. Observe that
$\Phi_2(\lambda \nu^*)=0$ if $|\lambda |=|\nu |$ and $\lambda \ne
\nu$. Indeed, since $\lambda\nu^* = (\lambda\lambda^*)(\lambda\nu^*)(\nu\nu^*)$, we have $\Phi_2(\lambda\nu^*) = (\lambda\lambda^*) \Phi_2(\lambda\nu^*) (\nu\nu^*) = 0$ because $D$ is commutative and $\lambda^*\nu=0$.

Finally, we are going to define the third conditional expectation
$\Phi_3$, from the commutative C*-algebra $D$ to its C*-subalgebra
$C_0(E^0)$. For this we need an explicit description of $D$. For $0\le r\le \infty$, let $E^r$ be the set of
(forward) paths in $E$ of length $r$, together with all paths of
length $\le r$ ending in a sink. We have truncation maps $\tau
_{r,s}\colon E^s\to E^r$, $\gamma \mapsto \gamma[r]$, for $r\le s\le
\infty$, where the truncation $\gamma[r]$ of $\gamma = e_1e_2\cdots$
is $e_1e_2 \cdots e_r$ (with $\gamma [0]=s(\gamma)$ and $\gamma
[r]=\gamma$ if $\gamma$ is a path of length $\le r$ ending in a
sink). For $r<\infty$, we put on $E^r$ the discrete topology.

Observe that $E^{\infty}$ is precisely the projective limit of the
inverse system
\begin{equation*}
\begin{CD}
\cdots @>\tau_{r,r+1}>> E^r @>\tau_{r-1,r}>> E^{r-1}
@>\tau_{r-2,r-1}>> \cdots @>\tau_{1,2}>> E^1 @>\tau_{0,1}>> E^0.$$
\end{CD}
\end{equation*}
We put on $E^{\infty}$ the inverse limit topology. A basis of
compact open sets for this topology is provided by the sets
$$U(\lambda)= \{\gamma \in E^{\infty}: \gamma [r]=\lambda\},$$
for $\lambda \in E^r$, $0\le r<\infty$. The maps $\tau_{r,s}$ are
continuous, proper, and surjective, and clearly $\tau _{r,s}\tau_{s,t}=\tau
_{r,t}$ for $r\le s\le t$.

By \cite{KPRR} (see also \cite{KPR}), $D=C_0(E^{\infty})$. We have
$$D=C_0(E^{\infty})=\varinjlim C_0(E^r).$$
We next define a positive integer $n_{\lambda}$ for each finite path
$\lambda $ in $E$. If the length of $\lambda $ is zero, then we set
$n_{\lambda} :=1$. If $\lambda =e_1\cdots e_t$ is a path of positive
length, we set
$$n_{\lambda}:= \prod_{i=1}^{|\lambda|} |s^{-1}s(e_i)| \,.$$

Let $\Phi ^t \colon C_0(E^t)\rightarrow C_0(E^0)$ be the map defined
as follows:
\begin{equation} \label{eq:keyformPhi1}
\Phi ^t(f)(v)= \sum _{\lambda \in E^t,\, s(\lambda )=v}
\frac{1}{n_{\lambda}} \,f(\lambda) \,,
\end{equation}
for $f\in C_0(E^t)$ and $v\in E^0$.

Using that $\sum _{\lambda \in E^t, s(\lambda)=v}
\frac{1}{n_{\lambda}}=1$ for every $0\le t<\infty$ and every $v\in
E^0$, one can easily check that  $\Phi ^t$ is a positive,
contractive linear map, and clearly $\Phi ^t(f)=f$ for every $f\in
C_0 (E^0)$. By Tomiyama's Theorem (see e.g. \cite[Theorem
1.5.10]{BO}), we get that $\Phi ^t$ is a conditional expectation for
all $t\ge 0$. Note that $\Phi ^t$ is faithful for all $t$.

We check now that the conditional expectations $\Phi ^t$ are
compatible with the maps in the inductive system. Let $\iota
_{t+1,t} \colon C_0(E^t)\to C_0(E^{t+1} )$ be the natural inclusion
map. For $f\in C_0 (E^t)$ and $v\in E^0$, we have
\begin{align*}
\Phi ^{t+1}(\iota_{t+1,t}(f))(v) & = \sum _{\lambda \in E^{t+1},\,
s(\lambda )=v} \frac{1}{n_{\lambda}}\, f(\lambda[t]) \\
& = \sum_{\substack{\gamma\in E^t,\, |\gamma|=t\\ s(\gamma )=v,\, r(\gamma)\notin \Si (E)}}
\frac{|s^{-1}r(\gamma)|}{n_{\gamma}|s^{-1}r(\gamma)|}\, f(\gamma) +
\sum_{\substack{\gamma\in E^t,\, |\gamma|\le t\\ s(\gamma )=v,\, r(\gamma)\in \Si (E)}}
\frac{1}{n_{\gamma}} \, f(\gamma)\\
& =\Phi ^t (f)(v)\, ,
\end{align*}
which proves that $\Phi ^{t+1}(\iota _{i+1,i}(f))=\Phi ^t(f)$ for
$f\in C_0 (E^t)$, as desired.

Since every $\Phi ^t$ is contractive and positive, we conclude that
there is a unique contractive, positive linear map $\Phi _3\colon
C_0(E^{\infty})\to C_0(E^0)$ extending all $\Phi ^t$'s. This map is
therefore a conditional expectation from $D=C_0(E^{\infty})$ onto
$C_0(E^0)$. We now observe that $\Phi _3$ is faithful. Indeed since
$D$ is a commutative C*-algebra of real rank zero, given any
positive nonzero element $a$ in $D$, there are a positive real number
$\epsilon $ and a nonzero projection $p$ in $D$ such that $\epsilon
\cdot p\le a$. Since $\Phi _3(p)\ne 0$ for all nonzero projections
$p$ in $D$, it follows that $\Phi _3$ is faithful.

In conclusion, we have obtained three faithful conditional
expectations $\Phi _i$, $i=1,2,3$, with
$$\begin{CD} C^*(E) @>\Phi
_1>> C^*(E)^{\alpha} @>\Phi_2>> D @>\Phi _3>> C_0(E^0)
\end{CD}$$
and so $\Phi_E:=\Phi _3\circ \Phi _2\circ \Phi _1$ is a faithful
conditional expectation from $C^*(E)$ onto $C_0(E^0)$.

It remains to check (\ref{basicPhi}). Let $\gamma $ and $\nu$ be two
(finite) paths in $E$ with $r(\gamma)= r(\nu)$. If $|\gamma|\ne |\nu |$, then $\Phi
_1(\gamma\nu ^*)=0$ and thus $\Phi _E(\gamma \nu ^*)=0$. If $|\gamma
|=|\nu |$ but $\gamma \ne \nu$ then
$$\Phi_E (\gamma \nu ^*)= \Phi _3(\Phi_2(\Phi _1(\gamma \nu ^*)))=
\Phi _3 (\Phi _2 (\gamma \nu^*))=\Phi _3(0)=0.$$

Finally, if $\gamma= e_1\cdots e_t$ is a path of length $t$ in $E$, then
$\gamma\gamma ^*$ corresponds to the characteristic function of $\{
\gamma \}$ in $C_0(E^t)$, and thus we get from
(\ref{eq:keyformPhi1}) that $$\Phi_E (\gamma \gamma ^*)=\Phi
_3(\gamma \gamma^*)=\biggl( \prod_{i=1}^t |s^{-1}s(e_i)|
\biggr)^{-1} s(\gamma)\, ,$$ establishing (\ref{basicPhi}) also in
this case.
\end{proof}

\begin{definition}  \label{remCondexp}
If $E$ is a row-finite graph, we call the conditional expectation
$\Phi_E$ of Theorem \ref{3condexp} the \emph{canonical conditional
expectation from $C^*(E)$ to $C_0(E^0)$}.
\end{definition}

%%%%%%%%%%%%%%%%%%%%%%%%%%%%%%%%%%%
\section{C*-algebras of separated graphs}
\label{Cstar}

Assume that $(E,C)$ is a separated graph.
In this section, we develop a characterization of $C^*(E,C)$ as an
amalgamated free product of ordinary graph C*-algebras. This will enable us to define the reduced
graph C*-algebra $\Cstred(E,C)$ when $(E,C)$ is finitely separated. We will show that for a trivially separated row-finite graph $E$, the reduced graph C*-algebra agrees with
the non-reduced one.

Set $A_0=C_0(E^0)= C^*(E^0,\emptyset)$. For each $X\in C$, consider the graph C*-algebra
$A_X=C^*(E_X)$, where $E_X$ is the subgraph of $E$ with $(E_X)^0=E^0$
and $(E_X)^1=X$. We have natural $*$-homomorphisms
$$A_0\to A_X\to C^*(E,C)$$
arising from the inclusions $(E^0,\emptyset) \rightarrow E_X$ and $(E_X,\{X\}) \rightarrow (E,C)$.

Let $\calC$ be a category, and consider an object $C_0$ in
$\calC$ and a family  $(C_{\iota})_{\iota \in I}$    of objects in $\calC$,
with morphisms $f_{\iota}\colon C_0\to C_{\iota }$. Then the \emph{amalgamated
coproduct of $(C_{\iota})_{\iota \in I}$ over $C_0$} is an object $C$ in $\calC$, together with morphisms $g_{\iota}\colon C_{\iota}\to C$ such
that $g_{\iota} \circ f_{\iota }=g_{\iota '}\circ f_{\iota '}$ for
all $\iota ,\iota '\in I$, which are universal in the following
sense: Given any other family of
morphisms $h_{\iota }\colon C_{\iota }\to D$ such that $h_{\iota }\circ f_{\iota}=h_{\iota '}\circ
f_{\iota'}$ for all $\iota ,\iota '\in I$, there is a unique
$h\colon :C\to D$ such that $h_{\iota }=h\circ g_{\iota }$ for all
$\iota \in I$.

We now show that $C^*(E,C)$ is an amalgamated coproduct of the C*-algebras $C^*(E_X)$. This is the same idea (and proof) as in Duncan's Theorem 1 \cite{Dun}, except that we express $C^*(E,C)$ as a coproduct of smaller algebras (but more of them) than Duncan uses.

\begin{proposition}
\label{prop:amalg} Let $(E,C)$ be a separated graph, and
consider $A_0=C_0(E^0)$ and $A_X=C^*(E_X)$ as above. Then $C^*(E,C)$,
together with the natural $*$-homomorphisms $f_X\colon A_X\to
C^*(E,C)$, is the amalgamated coproduct of the family $(A_X)_{X\in C}$ over the
C*-algebra $A_0$ in the category $\Cstaralg$.
\end{proposition}

\begin{proof}
We have to verify the universal property, so for $X\in C$ let
$h_X\colon A_X\to D$ be a $*$-homomorphism from $A_X$ to a
C*-algebra $D$ such that all compositions $A_0\to A_X\to D$ give the
same map $h_0$. We then have a family $(h_0(v))_{v\in E^0}$ of
orthogonal projections in $D$ and a family $(h_X (e))_{e\in X}$ of
partial isometries in $D$ for each $X\in C$, satisfying the
relations (V), (E), (SCK1), (SCK2). By the universal property of
$C^*(E,C)$, it follows that there exists a unique $*$-homomorphism
$h\colon C^*(E,C)\to D$ such that $h(v)=h_0(v)$ for all $v\in E^0$
and $h(e)=h_X(e)$ for all $e\in X$, for any $X\in C$. It follows
that $h\circ f_X=h_X$ for all $X\in C$, and so $C^*(E,C)$ is the
amalgamated coproduct of the family $(A_X)_{X\in C}$ over $A_0$.
\end{proof}

\begin{remark}
\label{rem:forrings} The same proof as above shows that $L(E,C)$ is
the amalgamated coproduct of the family $(L(E_X))_{X\in C}$ over the
$*$-algebra $L_0=\bigoplus_{v\in E^0} \C v$, in the category of
complex $*$-algebras.
\end{remark}

\begin{definition}  \label{def:redamprod}
Voiculescu defined in \cite{voicu} the \emph{reduced amalgamated
product} of a nonempty family $(A_{\iota}, \Phi _{\iota})_{\iota \in
I}$ of unital C*-algebras containing a unital subalgebra $A_0$
with conditional expectations $\Phi _{\iota}\colon A_{\iota}\to
A_0$. The reduced amalgamated product $(A,\Phi)$ is uniquely
determined by the following conditions:
\begin{enumerate}
\item $A$ is a unital C*-algebra, and there are unital $*$-homomorphisms
$\sigma _{\iota}\colon A_{\iota} \to A$ such that $\sigma_{\iota}|
_{A_0}= \sigma_{\iota'}|_{A_0}$ for all $\iota,\iota '\in I$.
Moreover the map $\sigma_{\iota}| _{A_0}$ is injective and we
identify $A_0$ with its image in $A$ through this map.
\item $A$ is generated by $\bigcup_{\iota \in I}
\sigma_{\iota}(A_{\iota})$.
\item $\Phi\colon A\to A_0$ is a
conditional expectation such that $\Phi \circ \sigma_{\iota}= \Phi
_{\iota}$ for all $\iota \in I$.
\item For $(\iota _1,\dots ,\iota _n)\in \Lambda (I)$ and $a_j\in \ker{\Phi_{\iota _j}}$
we have $\Phi (\sigma_{\iota _1}(a_1)\cdots \sigma _{\iota
_n}(a_n))=0$. Here, $\Lambda (I)$ denotes the set of all finite tuples
$(\iota _1,\dots ,\iota _n) \in \bigsqcup_{n=1}^\infty I^n$ such that $\iota _i\ne \iota
_{i+1}$ for $i=1,\dots ,n-1$.
\item If $c\in A$ is such that $\Phi (a^*c^*ca)=0$ for all $a\in A$,
then $c=0$.
\end{enumerate}

The \emph{full amalgamated product} $\bgast_{A_0}\, A_{\iota}$ is by
definition the amalgamated coproduct of the family $(A_{\iota})_{\iota \in I}$
over $A_0$ in the category of unital C*-algebras. By (1), there is
a unique $*$-homomorphism $\sigma \colon \bgast_{A_0}\, A_{\iota}\to A$
such that $\sigma_{\iota}=\sigma \circ f_{\iota}$ for all $\iota\in I$, where
$f_{\iota}\colon A_{\iota}\to \bgast_{A_0}\, A_{\iota}$ are the canonical
maps, and by (2) this map is surjective. We also have a canonical
map $\bgast^{\text{alg}}_{A_0}\, A_{\iota}\to
\bgast_{A_0}\, A_{\iota}$, where $\bgast^{\text{alg}}_{A_0}\, A_{\iota}$ denotes the
\emph{algebraic} amalgamated free product (which is the amalgamated coproduct of
$(A_{\iota})_{\iota \in I}$ over $A_0$ in the category of unital $\C$-algebras).

We now briefly recall the construction in \cite{voicu}. Let
$M_{\iota}=L^2(A_{\iota}, \Phi _{\iota})$ be the Hilbert
$A_0$-bimodule given by the GNS-construction, where the action of
$A_0$ on the left is given by restricting to $A_0$ the canonical
action of $A_{\iota}$ on $M_{\iota}$. We have $M_{\iota}= A_0\oplus
M_{\iota}^0$ as a Hilbert bimodule, and the Hilbert $A_0$-module
$M$ is defined by
$$M=A_0\oplus \bigoplus_{(\iota_1,\dots ,\iota _n)\in \Lambda(I)}
M^0_{\iota_1}\otimes _{A_0}\cdots \otimes _{A_0} M^0_{\iota_n}.$$

There are representations $\lambda _{\iota}\colon A_{\iota}\to
\calL (M)$ corresponding to the action of $A_\iota$ on terms with left hand factor $M_\iota$, see for instance \cite{voicu}, \cite{ivanov}, and
$\bgast_{A_0}^{\text{red}}\, A_{\iota}$ is defined as the C*-subalgebra
of $\calL (M)$ generated by $\cup _{\iota \in I} \lambda
_{\iota} (A_{\iota})$. We have a cyclic vector $\xi:=1_{A_0}$ in $M$
such that $a\cdot \xi=\hat{a}$ for all $a\in A_{\iota}$, where
$\hat{a}$ denotes the copy of $a\in A_{\iota}$ in
$M_{\iota}\subseteq M$.

If all the kernels of the GNS representations are $0$, then the maps
$\sigma _{\iota}$ are isometries, and we can identify each
$A_{\iota}$ with its image in $A$.
\end{definition}

\begin{subsec} \label{C*redprep}
{\bf Preparation}.
We are now going to define the reduced graph C*-algebra
$\Cstred(E,C)$ of the finitely separated graph $(E,C)$. For
a C*-algebra $A$, we will denote by $\Atil$ the minimal
unital C*-algebra containing $A$, that is the subalgebra of the
multiplier algebra $M(A)$ of $A$ generated by $A$ and $1_{M(A)}$.

Set $B_0=\Atil_0$, and $B_X=\Atil_X$ for $X\in C$, where, as above,
$A_0=C_0(E^0)$ and $A_X=C^*(E_X)$. Then the canonical conditional
expectation $\Phi _X := \Phi_{E_X} \colon A_X\to A_0$ constructed in
Section \ref{sect:cond-expecs} extends canonically to a conditional
expectation $\Phi _X\colon B_X\to B_0$ (see e.g. \cite[Proposition
2.2.1]{BO}). Since $\Phi _X\colon A_X\to A_0$ is faithful, it follows that its extension to $B_X$ is also
faithful. Now we consider the reduced amalgamated product $(B,
\Phi)$ of the family $(B_X, \Phi _X)_{X\in C}$. Since all the
conditional expectations $\Phi _X$ are faithful, it follows from
\cite[Theorem 2.1]{ivanov} that the canonical conditional
expectation $\Phi \colon B\to \Atil_0$ is faithful.
\end{subsec}

\begin{definition}
\label{def:redCstar} Let $(E,C)$ be a finitely separated graph, and
let $A_0,B_0, A_X, B_X$ be as defined above, for $X\in C$. Consider
the reduced amalgamated product $(B, \Phi)$ of the family $(B_X,
\Phi _X)_{X\in C}$. Then the \emph{reduced graph C*-algebra}
$\Cstred(E,C)$ is the C*-subalgebra of $B$ generated by
$\bigcup _{X\in C}A_X$ in $B$ (where we identify each $A_X$ with its
isometric image in $B$). Observe that there is a faithful canonical
conditional expectation $\Phi \colon \Cstred(E,C)\to A_0$, such that $\Phi|_{A_X}= \Phi_X$ for all $X\in C$.

As with $C^*(E,C)$ (cf.~Definition \ref{defC*sepgraph}), we use the same symbols to denote vertices and edges of $E$ as for their canonical images in $\Cstred(E,C)$.
\end{definition}

We do not address here the question of extending Definition \ref{def:redCstar} to a functor from $\FSGr$ to $\Cstaralg$. However, several natural maps related to this possible functor will be needed, as follows.

First, given a finitely separated graph $(E,C)$, observe that the natural images in $\Cstred(E,C)$ of the vertices and edges of $E$ satisfy the defining relations of the $\C$-algebra $L(E,C)$. Hence, there is a unique *-homomorphism $L(E,C) \rightarrow \Cstred(E,C)$ that sends all vertices and edges of $E$ to their canonical images in $\Cstred(E,C)$. We refer to this map as the \emph{canonical map} from $L(E,C)$ to $\Cstred(E,C)$. For the same reason, we obtain a \emph{canonical map} $C^*(E,C) \rightarrow \Cstred(E,C)$, and the canonical map $L(E,C) \rightarrow \Cstred(E,C)$ is the composition of the canonical maps $L(E,C) \rightarrow C^*(E,C) \rightarrow \Cstred(E,C)$.

Next, suppose that $E$ is a row-finite graph, viewed as a trivially separated graph $(E,C)$ where $C_v= \{s^{-1}(v)\}$ for all $v\in E^0\setminus \Si(E)$. We then define $\Cstred(E) := \Cstred(E,C)$. From the previous paragraph, we obtain a \emph{canonical map} $C^*(E) \rightarrow \Cstred(E)$. We prove in Theorem \ref{thm:red-prod-inj} that this map is an isomorphism.

The final canonical map we require is given in the following lemma.

\begin{lemma}
\label{lem:preparingmain1} Let $(F,D)$ be a complete subobject of an object
$(E,C)$ in $\FSGr$, such that $E^0=F^0$. Then there is a natural embedding of $\Cstred(F,D)$ into $\Cstred(E,C)$ such that $E^0\cap \Cstred(F,D)$ and $E^1\cap \Cstred(F,D)$
are sent to their natural images in $E^0\cap \Cstred(E,C)$ and $E^1\cap \Cstred(E,C)$.
\end{lemma}

\begin{proof} Write $A_0=C_0(E^0)=C_0(F^0)$, and denote by
$M$ and $M'$ the Hilbert $A_0$-modules corresponding to $(E,C)$ and
$(F,D)$ respectively. For $X\in D$, let $\lambda '_X\colon
\Atil_X\to \calL (M')$ be the canonical representation of
$\Atil_X$ on $M'$, and for $Y\in C$, let $\lambda _Y\colon
\Atil_Y\to \calL (M)$ be the canonical representation of
$\Atil_Y$ on $M$.

Let $B$ be the C*-subalgebra of $\Cstred(E,C)$ generated by
$\bigcup _{X\in D} \lambda_X(\Atil_X)$, and let $\Phi _B\colon B\to
A_0$ denote the restriction of $\Phi_{(E,C)}\colon \Cstred(E,C)\to
A_0$ to $B$. Note that $(B, \Phi _B)$ satisfies conditions (1)--(5)
of Definition \ref{def:redamprod} with respect to the family
$(\Atil_X, \Phi _X)_{X\in D}$ ((5) is trivially satisfied because
$\Phi _{(E,C)}$ is faithful). Since these properties determine the
reduced amalgamated product, we obtain an isomorphism $\varphi
\colon (\Cstred(F,D))^{\sim} \to B$ such that $\Phi _B\circ \varphi
=\Phi _{(F,D)}$ and $\varphi \lambda _X'=\lambda _X$ for all $X\in
D$. It follows that $\varphi$ restricts to an isomorphism from
$\Cstred(F,D)$ to the C*-subalgebra of $\Cstred(E,C)$ generated
by $\bigcup _{X\in D} \lambda _X (A_X)$.
\end{proof}

The proof of the following lemma is straightforward.

\begin{lemma}
\label{lem:diffrverices} Assume that $(F,D)$ is a complete subobject
of an object $(E,C)$ in $\FSGr$, such that $E^1=F^1$ and $C=D$. Then
$$\Cstred(E,C) \cong \Cstred(F,D)\times
C_0(E^0\setminus F^0). \qquad\qquad \square$$
\end{lemma}

We are now ready to establish one of our main results. In particular,
this provides an extension of \cite[Theorem 7.3]{Tomf07} to finitely
separated graphs. It implies that the linear basis of the dense subalgebra $L(E,C)$ explicitly exhibited in \cite[Corollary 2.8]{AG} is linearly independent in $C^*(E,C)$. Thus, the paths in $E$ are
linearly independent in $C^*(E,C)$, and the vertices of $E$  constitute a set
of pairwise orthogonal nonzero projections in $C^*(E,C)$.

\begin{theorem}
\label{thm:red-prod-inj} Let $(E,C)$ be a finitely separated graph.
\begin{enumerate}
\item The canonical map $L(E,C)\to \Cstred(E,C)$ is
injective, and hence so is the canonical map $L(E,C)\to
C^*(E,C)$.
\item If $E$ is a {\rm(}non-separated{\rm)} row-finite graph, then the canonical map
$C^*(E)\to \Cstred(E)$ is an isomorphism.
\end{enumerate}
\end{theorem}

\begin{proof} Throughout, set $A_0 := C^*(E^0,\emptyset)= C_0(E^0)$.

(1) We first consider the case where $E^0$ is finite. In this case,
$A_0$ is a commutative finite-dimensional C*-algebra, and $A_0=
L(E^0,\emptyset)$. Let $\psi \colon L(E,C)\to \Cstred(E,C)$ be the
canonical map, and set $L_X:= L(E_X)= A_0\oplus L_X^{{\rm o}}$ for
$X\in C$, where $L_X^{{\rm o}}=\ker{(\Phi _X)|_{L_X}}$.  We will
denote algebraic tensor products by $\odot $.

We have
$$L(E,C)\cong \bgast_{A_0}^{\text{alg}}\, L_X =A_0\oplus \bigoplus_{(X_1,\dots ,X_n)\in \Lambda(C)}
L^{{\rm o}}_{X_1}\odot _{A_0}\cdots \odot _{A_0} L^{{\rm o}}_{X_n}
\,,$$ and we want to show that $\psi$ embeds each of the terms
$L^{{\rm o}}_{X_1}\odot _{A_0}\cdots \odot _{A_0} L^{{\rm o}}_{X_n}$
into the corresponding $M^{{\rm o}}_{X_1}\otimes _{A_0} \cdots
\otimes _{A_0}M^{{\rm o}}_{X_n}$, where $M_X:=L^2(C^*(E_X),\Phi _X)$
for all $X\in C$. For $(X_1,\dots ,X_n)\in \Lambda (C)$ and $a_i\in
L^{{\rm o}}_{X_i}$, $i=1,\dots ,n$, we have
$$\psi (a_1\odot _{A_0} \cdots \odot_{A_0} a_n) 1_{A_0} = \hat{a}_1\otimes _{A_0} \cdots \otimes_{A_0} \hat{a}_n
\in M^{{\rm o}}_{X_1}\otimes _{A_0} \cdots \otimes _{A_0}M^{{\rm
0}}_{X_n} \,.$$
 Hence, it suffices to
show that, for $z$ in the algebraic tensor product $L^{{\rm
0}}_{X_1}\odot _{\C} \cdots \odot_{\C} L^{{\rm o}}_{X_n}$, we have
$\langle z,z \rangle =0$ if and only if $z$ belongs to the kernel
$K_n$ of the natural map
\begin{equation}
\label{eq:L-map1} L^{{\rm o}}_{X_1}\odot _{\C} \cdots \odot_{\C}
L^{{\rm o}}_{X_n}\longrightarrow L^{{\rm o}}_{X_1}\odot _{A_0}\cdots
\odot _{A_0} L^{{\rm o}}_{X_n} \,,
\end{equation}
cf.~the proof of \cite[Proposition 4.5]{Lance}. We proceed by
induction on $n$. If $n=1$ then the result follows from the fact
that $\Phi _{X_1}$ is faithful and $L_{X_1}\subseteq C^*(E_{X_1})$,
so that $L_{X_1}^{{\rm o}}\subseteq M^{{\rm o}}_{X_1}$. Assume that
$n>1$ and that $L^{{\rm o}}_{X_1}\odot _{A_0}\cdots \odot _{A_0}
L^{{\rm o}}_{X_{n-1}}$ embeds in $M^{{\rm o}}_{X_1}\otimes
_{A_0}\cdots \otimes _{A_0} M^{{\rm o}}_{X_{n-1}}$. The map in
(\ref{eq:L-map1}) is the composition of the linear maps
\begin{equation}
\label{eq:L-map2} (L^{{\rm o}}_{X_1}\odot _{\C} \cdots \odot_{\C}
L^{{\rm o}}_{X_{n-1}})\odot_{\C}L^{{\rm o}}_{X_n}\longrightarrow
(L^{{\rm o}}_{X_1}\odot _{A_0}\cdots \odot _{A_0} L^{{\rm
0}}_{X_{n-1}})\odot_{\C} L^{{\rm o}}_{X_n}
\end{equation}
and
\begin{equation}
\label{eq:L-map3} (L^{{\rm o}}_{X_1}\odot _{A_0}\cdots \odot _{A_0}
L^{{\rm o}}_{X_{n-1}})\odot_{\C} L^{{\rm o}}_{X_n}\longrightarrow
(L^{{\rm o}}_{X_1}\odot _{A_0}\cdots \odot _{A_0} L^{{\rm
0}}_{X_{n-1}})\odot_{A_0} L^{{\rm o}}_{X_n}.
\end{equation}
Write $N_0 :=L^{{\rm o}}_{X_1}\odot _{A_0}\cdots \odot _{A_0}
L^{{\rm o}}_{X_{n-1}}$. By the induction hypothesis, $N_0$ embeds in
the Hilbert $A_0$-module $N:=M^{{\rm o}}_{X_1}\otimes _{A_0}\cdots
\otimes _{A_0} M^{{\rm o}}_{X_{n-1}}$. The Hilbert $A_0$-module
$M^{{\rm o}}_{X_1}\otimes _{A_0}\cdots \otimes _{A_0} M^{{\rm
0}}_{X_{n}}$ is the interior tensor product $N\otimes _{A_0} M^{{\rm
0}}_{X_n}$, so that is the completion of the inner-product module
$(N\odot _{\C}M^{{\rm o}}_{X_n})/Y$, where
$$Y:=\{z\in N\odot _{\C}M^{{\rm o}}_{X_n}: \langle z,z \rangle =0\}$$
and $\langle \cdot ,\cdot \rangle$ is the sesquilinear form on
$N\odot _{\C}M^{{\rm o}}_{X_n}$ defined by
$$\langle n_1\otimes m_1,
n_2\otimes m_2 \rangle = \langle n_2, \phi(\langle n_1,m_1\rangle
)m_2 \rangle ,
$$
for $n_1,n_2\in N$, $m_1,m_2\in M^{{\rm o}}_{X_n}$, where $\phi
\colon A_0 \to \calL (M^{{\rm o}}_{X_n})$ is the map given by the
left action of $A_0$ on $M^{{\rm o}}_{X_n}$.

Now we follow the proof of \cite[Proposition 4.5]{Lance}. Assume
that
$$z=\sum _{i=1}^k x_i\otimes y_i\in N_0\odot _{\C} L^{{\rm
0}}_{X_n}\subseteq N\odot _{\C} M^{{\rm o}}_{X_n}$$
 satisfies that
$\langle z,z\rangle =0$. Let $x=(x_1,\dots ,x_k)\in N_0^k\subseteq
N^k$. As in  \cite[proof of 4.5]{Lance}, $N^k$ is a Hilbert
$M_k(A_0)$-module and we have
$$\langle z,z\rangle = \langle y, \phi ^{(k)}(X)y\rangle ,$$
where $y=(y_1,\dots ,y_n)\in (L^{{\rm o}}_{X_n})^k$ and $X=(\langle
x_i,x_j\rangle )=\langle x, x\rangle _{M_k(A_0)}$. Since $M_k(A_0)$
is a finite dimensional C*-algebra, there are a projection $E$
and a positive element $B$ in $M_k(A_0)$ such that
$$BX=E, \qquad XE=X .$$
It follows that $xE=x$ and that $\phi ^{(k)}(E)y=0$. This shows that
$z$ belongs to the subspace of $N_0\odot _{\C} L^{{\rm o}}_{X_n}$
generated by all elements of the form $na\otimes m-n\otimes \phi
(a)m$, $n\in N_0,m\in L^{{\rm o}}_{X_n}, a\in A_0$, that is, to the
kernel of the map (\ref{eq:L-map3}).

Finally, assume that $z\in (L^{{\rm o}}_{X_1}\odot _{\C} \cdots
\odot_{\C} L^{{\rm o}}_{X_{n-1}})\odot_{\C}L^{{\rm o}}_{X_n}$ is
such that $\langle z,z \rangle =0$. Let $\ol{z}$ be the image of $z$
under the map (\ref{eq:L-map2}). Then $\langle \ol{z},\ol{z}
\rangle=0$ and by what we have just proven,
$$\ol{z}=\sum _j (\ol{z_j}a_j\otimes  y_j - \ol{z_j}\otimes a_j y_j)$$
for some $z_j\in L^{{\rm o}}_{X_1}\odot _{\C} \cdots \odot_{\C}
L^{{\rm o}}_{X_{n-1}}$, $y_j\in L^{{\rm o}}_{X_n}$, $a_j\in A_0$. It
follows that
$$z-\sum _j (z_ja_j\otimes  y_j - z_j\otimes a_j y_j)\in K_{n-1}\odot_{\C} L^{{\rm o}}_{X_n}\subseteq K_n.$$
Since $z_ja_j\otimes  y_j - z_j\otimes a_j y_j\in K_n$ for all $j$,
we conclude that $z\in K_n$, as desired.

This concludes the proof in the case where $E^0$ is finite. If $E^0$
is infinite, then by \cite[Proposition 3.6]{AG} we can write $L(E,C)=\varinjlim L(F,D)$, where
$(F,D)$ ranges over all the finite complete subobjects of $(E,C)$,
and all the limit maps $L(F,D)\to L(E,C)$ are injective.

For a finite complete subobject $(F,D)$ of $(E,C)$, the canonical map
$L(F,D)\to \Cstred(F,D)$ is injective, as proved above. Let
$F'$ be the subgraph of $E$ with $(F')^0= E^0$ and $(F')^1= F^1$. Then
the canonical map $\Cstred(F,D)\to \Cstred(E,C)$ is the
composition of the canonical maps $\Cstred (F,D)\to \Cstred
(F', D)$ and $\Cstred (F',D)\to \Cstred (E,C)$. By
Lemmas \ref{lem:diffrverices} and \ref{lem:preparingmain1}, both of the latter maps
are injective and so the canonical map $L(F,D)\to \Cstred(E,C)$ is
also injective. Since $L(E,C)=\varinjlim L(F,D)$, it follows that
the canonical map $L(E,C)\to \Cstred(E,C)$ is
injective, as desired.

(2) Since $E$ is a non-separated graph, we identify $C$ with
$E^0\setminus \Si (E)$, by corresponding $\{s^{-1}(v)\}$ to $v$ for
non-sinks $v\in E^0$. We shall write $E_v= E_{s^{-1}(v)}$ for $v\in
C$. Set $n_v := |s^{-1}(v)|$ and $L_v:=L(E_v)$, and set
$A_v:=C^*(E_v)$, and  $B_v:=\Atil_v$. Let $B= {C^*(E)}^{\sim}$ and
recall that we have a faithful conditional expectation $\Phi \colon
B\to \Atil_0$ (Subsection \ref{C*redprep}). To establish the desired
isomorphism, it is enough to show that $(B,\Phi)$ satisfies
conditions (1)--(5) of Definition \ref{def:redamprod}, because these
conditions characterize completely the reduced amalgamated product
of the family $(B_v,\Phi _v)_{v\in C}$. All the conditions are
immediate, with the exception of condition (4). To show condition
(4), take a sequence of vertices $v_1,\dots ,v_n$ in $C$, with
$n\ge2$, such that $v_i\ne v_{i+1}$ for $i=1,\dots ,n-1$. We have to
show that $\Phi (a_1a_2\cdots a_n)=0$ when $a_i\in \ker \Phi_{v_i}$
for $i=1,\dots ,n$. Since $L_{v_i}^{{\rm o}} := L_{v_i} \cap \ker
\Phi_{v_i}$ is dense in $\ker \Phi_{v_i}$, it suffices to prove this
statement for all choices of $a_i\in L_{v_i}^{{\rm o}}$, $i=1,\dots
,n$.

Consider $v\in C$, and note that any path of positive length in $E_v$ consists of either a sequence of loops at $v$ or else a sequence of loops at $v$ followed by one edge from $v$ to a different vertex. In particular, all paths of positive length in $E_v$ start at $v$. Observe that every element of $L_v$ is a linear combination of terms of the following five types:
\begin{enumerate}
\item Paths $\gamma$ in $E_v$ of positive length.
\item Paths $\nu^*$, where $\nu$ is a path in $E_v$ of positive length.
\item Paths $\gamma\nu^*$, where $\gamma$ and $\nu$ are distinct paths in $E_v$ of positive length.
\item Terms $\gamma(ee^*- n_v^{-1}v)\gamma^*$, where $e\in s^{-1}(v)$ and $\gamma$ is a path in $E_v$ from $v$ to $v$.
\item Vertices $w\in E^0$.
\end{enumerate}
All terms of types (1)--(4) are in $L^{{\rm o}}_v$ (recall formula
\eqref{basicPhi}), whereas $\Phi_v(w)=w$ for $w\in E^0$. Hence, the
terms of types (1)--(4) span $L^{{\rm o}}_v$.

Returning to our previous discussion, we see that it is enough to
show that $\Phi (a_1a_2\cdots a_n)=0$ for all choices of $a_i \in
L^{{\rm o}}_{v_i}$ where each $a_i$ has one of the forms (1)--(4).
We may also assume that  $a_1a_2\cdots a_n \ne 0$. It is easy to
verify the following:
\begin{itemize}
\item If $a_i$ has one of the forms (1), (3), or (4) and $i>1$, then $a_{i-1}$ is necessarily of type (1).
\item If $a_i$ has one of the forms (2), (3), or (4) and $i<n$, then $a_{i+1}$ is necessarily of type (2).
\end{itemize}
It follows that at most one $a_i$ can be of type (4). If such a term occurs, then
$$a_1a_2\cdots a_n = \gamma(ee^*- n_v^{-1}v)\nu^*$$
for some $v\in E^0$, some $e\in s^{-1}(v)$, and some paths $\gamma$, $\nu$ in $E$ that end at $v$. In this case, it is clear that $\Phi (a_1a_2\cdots a_n)=0$. (Consider the cases $\gamma=\nu$ and $\gamma\ne\nu$ separately.)

If no $a_i$ is of type (4), then one of the following holds: $a_1a_2\cdots a_n = \gamma$ for some path $\gamma$ in $E$ of positive length; or $a_1a_2\cdots a_n = \nu^*$ for some path $\nu$ in $E$ of positive length; or
$$a_1a_2\cdots a_n =  \gamma_1\gamma_2 \cdots\gamma_j \nu^*_k \nu^*_{k+1} \cdots \nu^*_n$$
where $k=j$ or $k=j+1$, and each $\gamma_i$ or $\nu_i$ is a path of positive length in $E_{v_i}$. Obviously $\Phi (a_1a_2\cdots a_n)=0$ in the first two cases, and it holds in the third case provided $\gamma_1\gamma_2 \cdots\gamma_j \ne \nu_n\nu_{n-1} \cdots\nu_k$. Thus, it suffices to assume that the third case obtains, and that $\gamma_1\gamma_2 \cdots\gamma_j = \nu_n\nu_{n-1} \cdots\nu_k$, and to derive a contradiction.

We cannot have $j=1$ and $k=n$, since then $n=2$ while $\gamma_1$ and $\nu_2$ have different starting vertices. We cannot have $j=1$ and $k<n$, since $\gamma_1$ only changes vertices on its terminal edge, whereas $\nu_n$ must change vertices once, and the following path $\nu_{n-1}$ has at least one edge. Thus $j>1$, and similarly $k<n$. Since $\gamma_1\gamma_2\ne 0$, we have $r(\gamma_1)= v_2\ne v_1= s(\gamma_1)$, so $\gamma_1$ consists of a sequence of loops at $v_1$ followed by an edge from $v_1$ to $v_2$. Similarly, $\nu_n$ consists of a sequence of loops at $v_n$ followed by an edge from $v_n$ to $v_{n-1}$. Thus, since $\gamma_1\gamma_2 \cdots\gamma_j = \nu_n\nu_{n-1} \cdots\nu_k$, we see that $\gamma_1=\nu_n$. Now $\gamma_2\cdots\gamma_j= \nu_{n-1}\cdots\nu_k$, and we can continue in the same manner. We eventually find that $n-k+1=j$ and $\gamma_j=\nu_k$. However, $\gamma_j \ne \nu_{j+1}$ because these paths have different starting vertices, and $\gamma_j\ne \nu_j$ (in c
 ase $k=j$) by the assumption of type (3) for $a_j= \gamma_j\nu^*_j$. This provides the desired contradiction.
\end{proof}

Suppose $V$ is a subset of $E^0 \setminus \Si(E)$ and $\calX : V \to C$ is
a function such that $\calX(v) \in C_v$ for every $v\in V$. Define a
subgraph $E_{\calX}$ of $E$ so that $E^0_{\calX} = E^0$ and $E^1_{\calX} =
\bigsqcup_{v\in V} \calX(v)$. View $E_{\calX}$ as a trivially separated
graph, and note that it is a complete subobject of $(E,C)$.

\begin{corollary}  \label{cor:C*partialchoice}
For any $V$ and $\calX$ as above, the induced map $C^*(E_{\calX}) \to
\Cstred(E_{\calX}) \to \Cstred(E,C)$ is injective, and hence so is the
canonical map $C^*(E_{\calX}) \to C^*(E,C)$.
\end{corollary}

\begin{proof} By Theorem \ref{thm:red-prod-inj}(2), we have $C^*(E_{\calX})\cong
\Cstred(E_{\calX})$, and so Lemma
\ref{lem:preparingmain1} gives that the canonical map $C^*(E_{\calX})\to \Cstred(E,C)$ is injective.
\end{proof}

\begin{remark}
\label{rm:exact} By \cite[Corollary 4.5.4]{BO}, every graph
C*-algebra of a row-finite graph is nuclear. It follows from this
and \cite[Corollary 4.8.3]{BO} that the reduced C*-algebra
$\Cstred(E,C)$ is exact for every finitely separated graph
$(E,C)$. That $\Cstred(E,C)$ is not nuclear in general
follows from the example
$$\Cstred(E,C) \cong \Cstred(\mathbb F_2),$$
where $(E,C)$ is the separated graph with one vertex $v$, two
edges $e_1$, $e_2$, and $C=\{\{e_1\},\{e_2\}\}$. (Recall that
$\Cstred(\mathbb F_2)$ is not nuclear because $\mathbb F
_2$ is not amenable \cite[Theorem 2.6.8]{BO}.)
\end{remark}

We recall the following definitions, see e.g. \cite{AG, Tomf07}.

\begin{definition} \label{defCSsat} Let $(E,C)$ be a finitely
separated graph. Recall the relation $\ge$ defined on $E^0$ by
setting $v\geq w$ if and only if there is a path $\mu$ in $E$ with
$s(\mu)=v$ and $r(\mu)=w$. A subset $H$ of $E^0$ is called
\emph{hereditary} if $v \geq w$ and $v\in H$ always imply $w\in H$.
The set $H$ is called \emph{saturated} if $r(s^{-1}(v)) \subseteq H$
implies $v\in H$ for any $v\in E^0$ which is not a sink or an
infinite emitter. Finally, $H$ is called \emph{$C$-saturated} if
$r(X) \subseteq H$ for some $X\in C_v$, $v\in E^0$, implies $v\in
H$.\end{definition}

Let $\calH$ be the lattice of hereditary $C$-saturated subsets of
$E^0$. By \cite[Theorem 6.11]{AG} there is a lattice isomorphism
between $\calH$ and the lattice $ \Tr(A)$ of two-sided ideals of
$L(E,C)$ generated by idempotents. In the C*-algebra case, we
are at least able to show that the analogous map $\calH\to \calL (C^*(E,C))$
is injective. Here, for any C*-algebra $A$, we denote by $\calL
(A)$ the lattice of closed (two-sided) ideals of $A$. For a subset
$X$ of $E^0$ we denote by $I(X)$ the closed ideal of $C^*(E,C)$
generated by $X\cap C^*(E,C)$.

\begin{corollary}
\label{Mlattinclusion} Let $(E,C)$ be a finitely separated graph,
and let $\calH$ be the lattice of hereditary, $C$-saturated subsets
of $E^0$. Then there is an order-embedding $\calH\to \calL (C^*(E,C))$,
given by $H\mapsto I(H)$. \end{corollary}

\begin{proof} Clearly, it suffices to show that $E^0\cap I(H) =H$, for any $H\in \calH$.
Thus, let $H$ be a hereditary $C$-saturated subset of $E^0$.

We construct a finitely separated graph $(E/H, C/H)$ as in
\cite[Construction 6.8]{AG}. Namely, $E/H$ is the quotient graph,
that is, the subgraph of $E$ with
$$(E/H)^0 = E^0\setminus H \qquad\quad \text{and} \qquad\quad (E/H)^1= r_E^{-1}(E^0\setminus H) = E^1/H,$$
and, for $v\in (E/H)^0$,  we set
$$(C/H)_v := \{ X/H \mid X\in C_v  \},$$
which is a partition of $s_{E/H}^{-1}(v)$, and $C/H :=
\bigsqcup_{v\in E^0\setminus H } (C/H)_v$. Here, for any $X\subseteq
E^1$, we denote by $X/H$ the set $\{e\in X: r(e)\notin H\}$. Observe
that $X/H\ne \emptyset $ for all $X\in C_v$ with $v\in E^0\setminus
H$, because $H$ is $C$-saturated.

Since $r_E^{-1}(H)\cap C^*(E,C) \subseteq I(H)$, the cosets of the elements in $(E/H)^0\cap C^*(E,C)$ and $(E/H)^1\cap C^*(E,C)$ generate $C^*(E,C)/I(H)$. It is easily checked (by using the universal property of $C^*(E,C)$) that $C^*(E,C)/I(H)$ is presented by the above generators together with the defining relations of $C^*(E/H,C/H)$. Thus, we obtain an isomorphism
$$C^*(E,C)/I(H)\longrightarrow C^*(E/H, C/H)$$
sending $v+I(H) \mapsto v$ for $v\in (E/H)^0$ and $e+I(H)\mapsto e$ for $e\in (E/H)^1$. Now any vertex $v\in E^0\setminus H$ is nonzero as an element of $L(E/H, C/H)$ (cf.~\cite[Corollary 2.8]{AG}). Since
$L(E/H,C/H)$ embeds naturally in $ C^*(E/H, C/H)$ by Theorem
\ref{thm:red-prod-inj}(1), it follows that $v\notin I(H)$. Therefore $E^0\cap I(H) =H$, as desired.
\end{proof}

%%%%%%%%%%%%%%%%%%%%%%%%%%%%%%%%%%%%%%%
\section{Simplicity in $\Cstred(E,C)$}
\label{sect:idealsinCred}

For a finitely separated graph $(E,C)$, the reduced C*-algebra $\Cstred(E,C)$ has typically fewer ideals than the full C*-algebra $C^*(E,C)$. In fact, it can easily happen that $\Cstred(E,C)$ is simple while $C^*(E,C)$ is not. We shall consider the two
main examples from \cite{AG} and a related one, and we will show that the corresponding
reduced graph C*-algebras are indeed simple. These are somewhat
exotic examples of infinite simple C*-algebras; for instance, one has stable rank one but not real rank zero (see Corollary \ref{cor:C*Enn}). We do not know whether the others are purely infinite or have real rank zero.

We start by taking examples with only one vertex. The main tool is
the following result of Avitzour (\cite[Proposition 3.1]{avitzour}).
Since we will only use the case of faithful states, we state
below the result in this case.

\begin{proposition}
\label{prop:avitzour} \cite[Proposition 3.1]{avitzour} Let $A$, $B$ be
unital C*-algebras and $\phi$, $\psi$ faithful states on them.
Let $(D,\Phi)$ be the reduced amalgamated product of $(A,\phi)$ and
$(B,\psi)$ {\rm(}over $\C${\rm)}. Let $a\in \ker\phi$ and $b\in \ker\psi$ be unitaries
such that $\phi$, $\psi$ are invariant with respect to conjugation
by $a$, $b$ respectively. Let $c\in \ker\psi$ be a unitary such that
$\psi (b^*c)=0$.

Then for all $x$ in $D$,
$$\Phi (x)\in {\rm \overline{co}} \{ u^*xu: u \text{ unitary}\},$$
where ${\rm \overline{co}}$ denotes the norm-closed convex hull. It is
enough to take $u$ in the group generated by $a$, $b$, $c$.
\end{proposition}

It follows readily from this result that in the given situation, $D$ must be simple.
Indeed, let $J$ be a nonzero closed ideal of $D$, and let $x$ be a
nonzero positive element of $J$. Since $\phi$ and $\psi$ are
faithful it follows from \cite{dykema} or \cite{ivanov} that $\Phi$
is faithful and so Proposition \ref{prop:avitzour} gives that $J$
contains the invertible element $\Phi (x)$.

We apply now the result to reduced graph C*-algebras.

As in \cite{McCla3}, we will use the following unitaries in
$M_n(\C)$. Let $\lambda _n$ be a primitive $n$-th root of $1$, and
set:

$$u_n :=\diag(1,\lambda _n,\dots ,\lambda _n^{n-1}),\qquad\qquad
v_n :=\begin{pmatrix} 0 & 1 & 0 & \cdots & 0\\
0 & 0 & 1 & \cdots & 0\\
& \cdots & & \cdots &  \\
0 & 0 & \cdots & 0 & 1\\
1 & 0 & 0 & \cdots & 0
\end{pmatrix} \,.
$$

\medskip

\begin{proposition}
\label{prop:twoLeavialgs} Let $n,m>1$, and let $(E,C)$ be the separated
graph with one vertex $v$ and with $C_v :=\{X,Y\}$, where $|X|=n$ and
$|Y|=m$. Then the reduced graph C*-algebra $\Cstred(E,C)$
is simple.
\end{proposition}

\begin{proof}
Set $A :=\calO_n$ and $B :=\calO _m$, where as usual $\mathcal
O_k$ denotes the Cuntz algebra, and identify $A=C^*(E_X)$ and $B=C^*(E_Y)$. Then $(\Cstred(E,C), \Phi)$ is the reduced amalgamated
product of $(\calO_n,\phi_n)$ and $(\calO _m,\phi _m)$,
where we denote by $\phi _k$ the canonical faithful
state on $\calO_k$. There is a
standard copy of $M_n(\C)$ in $\calO_n$, namely the
linear span of $\{ef^*: e,f\in X\}$, and using this copy we
define the unitary $a :=v_n$ in $A$. Notice that $\phi _n$ is the
composition of the canonical conditional expectation from $\mathcal
O_n$ onto the $AF$-algebra $\calO_n^{\alpha}$ and the tracial
state $\tau _n$ on $\calO_n^{\alpha}=\varinjlim M_{n^i}(\C)$, where $\alpha$ denotes the gauge action.
Using this it is quite easy to show that $\phi_n (axa^*)=\phi_n (x)$ for
all $x\in \calO_n$. Indeed,  $\phi _n$ is invariant with
respect to conjugation by any unitary in $\calO_n^{\alpha}$. Observe that $\phi _n(a)= \trace(v_n) =0$.

Similarly, $b :=v_m$ and $c :=u_m$ are unitaries in $\ker\phi_m$, and $\phi _m$
is invariant with respect to conjugation by both $b$ and $c$.
Moreover, $\phi _m(b^*c)=0$. It therefore follows from Proposition
\ref{prop:avitzour} that $\Cstred(E,C)$ is a simple
C*-algebra.
\end{proof}

We need for our next examples a slight generalization of Proposition
\ref{prop:avitzour} for reduced amalgamated products over C*-algebras different from $\C$. Other generalizations to
this context have been obtained in \cite{McCla3} and \cite{ivanov}.

\begin{proposition}
\label{fullPavitzour} Let $A$, $B$, $A_0$ be unital C*-algebras with $A_0\subseteq A$ and $A_0\subseteq B$, and let $\phi: A\rightarrow A_0$ and $\psi: B\rightarrow A_0$ be faithful conditional expectations. Let $(D,\Phi)$ be the reduced amalgamated product of $(A,\phi)$ and $(B,\psi)$, and let $\pi \colon A \,\bgast^{\rm{alg}}_{A_0}\, B\to D$ be the natural map from the algebraic amalgamated product to $D$.

Assume there is a central projection $P\in A_0$ such
that $PA_0=\C P$. Let $a\in P(\ker\phi)P$ and $b\in P(\ker\psi)P$ be unitaries in
$PAP$ and $PBP$ respectively, such that $\phi|_{PAP}$, $\psi|_{PBP}$
are invariant with respect to conjugation by $a$, $b$ respectively.
Let $c\in P(\ker\psi)P$ be a unitary in $PBP$ such that $\psi (b^*c)=0$.
Then for all $x$ in $\pi(P)D\pi(P)$,
$$\Phi (x)\in {\rm \overline{co}} \{ u^*xu: u \text{ unitary in } \pi (P)D\pi (P)\}.$$
It is
enough to take $u$ in the group generated by $a,b,c$.
\end{proposition}

\begin{proof}
The proof follows the steps of that of \cite[Proposition
3.1]{avitzour}. Let us just mention what are the main steps. Let $M$ be the Hilbert $A_0$-module arising in the construction of $D$ (recall Definition \ref{def:redamprod}). We identify $\pi$ with the standard representation $A \,\bgast^{\text{alg}}_{A_0}\, B\to \calL_{A_0}(M)$.

Let $W_0\subseteq PA \,\bgast^{\text{alg}}_{A_0}\, B$ be the span of those words starting
with an element from $P\ker\phi$ or from the constants $\C P$ or a multiple of $b$. Let $W_1$
be the span of those words starting with some $b'$ in $P\ker\psi$ such that
$\psi (b^*b')=0$. Let
$$H_i= \ol{\pi (W_i)\xi }\subseteq \pi (P)M.$$
Then  $\pi (P)M= H_0\oplus H_1$. Since for $x\in \pi (P)D\pi (P)$
we have
$$\| x \|_{\calL _{A_0}(M)}=\| x|_{\pi (P)M}\|_{\calL_{A_0}(\pi (P)M)} \,, $$
we can apply the proof of \cite[Proposition 3.1]{avitzour} to show
that $$\Phi (x)\in {\rm \overline{co}} \{ u^*xu: u \text{ unitary in
} \pi (P)D\pi (P)\},$$ as desired.
\end{proof}

\begin{corollary}
\label{fullPavitzourcorolar} Assume that the conditions of the above
proposition hold, and that in addition $\pi (P)$ is a full
projection in $D$. Then $D$ is simple.
\end{corollary}

\begin{proof} First, recall from \cite[Theorem 2.1]{ivanov} that $\Phi$ is faithful.
Let $J$ be a nonzero closed ideal of $D$. Since $\pi (P)$ is a full
projection in $D$, we have that $\pi (P)J\pi (P)$ is nonzero. Let
$x$ be a nonzero positive element in $\pi (P)J\pi (P)$. Then $\Phi(x)$ is a nonzero scalar multiple of $P$, and so it
follows from Proposition \ref{fullPavitzour} that $\pi (P)\in J$.
Since $\pi (P)$ is full in $D$, we conclude that $J=D$.
\end{proof}

The next example is related to an example considered by McClanahan
in \cite[Example 3.12]{McCla3} (see Proposition \ref{isoMcClan1} for the precise relationship). However we use in the proof our version of
Avitzour's result (Proposition \ref{fullPavitzour}), which is
simpler than the one used in \cite{McCla3}.

\begin{example}  \label{ex:EmnCmn}
For integers $1\le m\le n$, define the separated graph
$(E(m,n),C(m,n))$, where
\begin{enumerate}
\item $E(m,n)^0 := \{v,w\}$ (with $v\ne w$).
\item $E(m,n)^1 :=\{\alpha_1,\dots , \alpha_n,\beta_1,\dots ,\beta_m\}$ (with $n+m$ distinct edges).
\item $s(\alpha_i)=s(\beta _j) =v$ and $r(\alpha _i)=r(\beta _j)=w$
for all $i$, $j$.
\item $C(m,n)= C(m,n)_v := \{X,Y\}$, where $X:= \{\alpha_1,\dots ,\alpha_n\}$ and $Y:=  \{\beta _1,\dots, \beta _m \}$.
\end{enumerate}
By \cite[Proposition 2.12]{AG}, $L(E(m,n),C(m,n))\cong M_{n+1}(L(m,n)) \cong M_{m+1}(L(m,n))$, where $L(m,n)$ is the classical Leavitt algebra of type $(m,n)$. The same argument (by way of universal properties) shows that
\begin{equation}  \label{isomatUnc}
C^*(E(m,n),C(m,n)) \cong M_{n+1}(U^{\text{nc}}_{m,n})\cong
M_{m+1}(U^{\text{nc}}_{m,n}) \,,
\end{equation}
where $U^{\text{nc}}_{m,n}$ denotes the C*-algebra generated by the entries of a universal unitary $m\times n$ matrix, as studied by Brown and McClanahan in \cite{Brown, McCla1, McCla2, McCla3}.
\end{example}

The reduced graph C*-algebra of $(E(m,n),C(m,n))$ is Morita equivalent to McClanahan's example, as we will show in Section \ref{sect:relMcClan}. (Hence, the following proposition can also be obtained as a corollary of McClanahan's results.)

\begin{proposition}
\label{prop:simple2} Let $1<m\le n$, and let $(E,C) := (E(m,n),C(m,n))$ be the separated graph described in Example {\rm\ref{ex:EmnCmn}}. Then the reduced graph
C*-algebra $\Cstred(E(m,n),C(m,n))$ is simple.
\end{proposition}

\begin{proof} Set $A_0:= \C v\oplus \C w$, and identify $A_0$ with $\C^2$ so that $v$ and $w$ correspond to $(1,0)$ and $(0,1)$ respectively. Set $A:= C^*(E_X)$ and $B:= C^*(E_Y)$, and identify $A$ and $B$ with $M_{n+1}(\C)$ and $M_{m+1}(\C)$ so that $v$ and $w$ correspond to $\diag(1,\dots,1,0)$ and $\diag(0,\dots,0,1)$ in each case. The canonical conditional
expectations $\Phi_X$ and $\Phi _Y$ are easily seen to correspond to
the maps $\phi\colon M_{n+1}(\C)\to \C^2$ and $\psi\colon
M_{m+1}(\C)\to \C^2$ given by
$$\phi \bigl( [a_{ij}] \bigr)= \biggl( \frac{1}{n}\sum _{i=1}^n a_{ii},\, a_{n+1,n+1} \biggr),\qquad
\psi \bigl( [b_{ij}] \bigr)= \biggl( \frac{1}{m}\sum _{j=1}^m b_{jj},\, b_{m+1,m+1} \biggr).$$

Take $P:=v$, and observe that $PA_0=\C P$ and that $P$ is a full
projection in both $A$ and $B$, so certainly $P$ will be a full
projection in $D:=A \,\bgast_{A_0}^{\text{red}}\, B=\Cstred(E,C)$.
Consider the unitaries $a :=\diag(v_n,0)$ in $M_{n+1}(\C)$ and
$b :=\diag(v_m,0)$, $c:=\diag(u_m,0)$ in $M_{m+1}(\C)$ respectively; then $a\in PAP$ and $b,c\in PBP$ with the above identifications. We have
\begin{equation}  \label{kerphipsi}
\phi (a)=\psi (b)=\psi (c)= \psi (b^*c)=0,
\end{equation}
and moreover $\phi|_{PAP}$, $\psi|_{PBP}$ are invariant with respect
to conjugation by $a$, $b$ respectively, and thus the conditions in
Proposition \ref{fullPavitzour} are satisfied. It follows from
Corollary \ref{fullPavitzourcorolar} that $\Cstred(E,C)$ is
a simple C*-algebra.
\end{proof}

\begin{remark}  \label{C*E1nC1n}
To fill in the cases not covered by Proposition \ref{prop:simple2}, let $n\ge1$ and consider  $(E,C):= (E(1,n),C(1,n))$. If $n>1$, then $U^{\text{nc}}_{1,n} \cong \calO_n$ and \eqref{isomatUnc} implies that
$$C^*(E,C) \cong M_2(\calO_n) \cong M_{n+1}(\calO_n),$$
whence $C^*(E,C)$ is simple. In this case, the full and reduced C*-algebras of $(E,C)$ coincide, and $\Cstred(E,C)$ is again simple.

Since $U^{\text{nc}}_{1,1} \cong C(\T)$, the case $m=n=1$ reduces to $C^*(E,C)\cong M_2(C(\T))$ by \eqref{isomatUnc}. Following the construction in the proof of \cite[Proposition 2.12]{AG}, there is an explicit isomorphism
$$\psi: C^*(E,C) \longrightarrow M_2(C(\T)) = M_2(\C)\otimes_{\C} C(\T)$$
sending
$$v\longmapsto e_{11} \,,\qquad w\longmapsto e_{22} \,,\qquad \alpha\longmapsto ze_{12} \,,\qquad \beta\longmapsto e_{12} \,,$$
where $z$ is the canonical unitary generator of $C(\T)$. We shall use this isomorphism to see that $C^*(E,C)= \Cstred(E,C)$. Thus, the case $m=n=1$ is the only one for which $\Cstred(E,C)$ is not simple.

Identify $A:= C^*(E_X)$ and $B:= C^*(E_Y)$ with their canonical images in $C^*(E,C)$, and set $A_0 := \C v\oplus \C w$. There is a faithful conditional expectation
$$\phi\otimes\tau : M_2(\C) \otimes_{\C} C(\T) \longrightarrow \begin{pmatrix} \C&0\\ 0&\C \end{pmatrix} \,,$$
where $\phi \left( \begin{smallmatrix} a&b\\ c&d \end{smallmatrix} \right)= \left( \begin{smallmatrix} a&0\\ 0&d \end{smallmatrix} \right)$ for $\left( \begin{smallmatrix} a&b\\ c&d \end{smallmatrix} \right) \in M_2(\C)$ and $\tau$ is the canonical faithful trace on $C(\T)$. This corresponds to a faithful conditional expectation $\Phi: C^*(E,C) \rightarrow A_0$, because $\psi$ restricts to an isomorphism of $A_0$ onto $\left( \begin{smallmatrix} \C&0\\ 0&\C \end{smallmatrix} \right)$. We claim that $(C^*(E,C),\Phi)$ satisfies the conditions of Definition \ref{def:redamprod} to be the reduced amalgamated product of $(A,\Phi_X)$ and $(B,\Phi_Y)$. Conditions (1) and (2) are clear, (3) is easily checked, and (5) follows from the faithfulness of $\Phi$. To check (4), observe that since $\ker\Phi_X= \C\alpha\oplus \C\alpha^*$ and $\ker\Phi_Y= \C\beta\oplus \C\beta^*$, it suffices to show that $\Phi$ vanishes on all finite paths of the forms
$$\alpha\beta^*\alpha\beta^*\cdots \,,\qquad \alpha^*\beta\alpha^*\beta\cdots \,,\qquad \beta\alpha^*\beta\alpha^*\cdots \,,\qquad \beta^*\alpha\beta^*\alpha\cdots \,.$$
However, $\psi$ maps these paths to products of the form $z^ke_{ij}$ or $(z^*)^ke_{ij}$ with $k\ge1$, and $\phi\otimes\tau$ vanishes on such products.

Therefore $(C^*(E,C),\Phi)= (A,\Phi_X) \,\bgast^{\text{red}}_{A_0}\, (B,\Phi_Y)$ and so $C^*(E,C)= \Cstred(E,C)$ in this case, as claimed.
\end{remark}

Finally, we show with another example that the structure of
hereditary $C$-saturated subsets of $E^0$ is not respected in
$\Cstred(E,C)$ in general, that is, there can be two
different hereditary $C$-saturated subsets $H_1$ and $H_2$ which
generate the same ideal of $\Cstred(E,C)$. This heavily
contrasts with the situation for the full graph C*-algebra
$C^*(E,C)$.

Let $k,l,m,n\ge2$ be integers. Consider the separated graph $(E,C)$, where
\begin{enumerate}
\item $E^0 := \{v,w_1,w_2\}$ (with $3$ distinct vertices).
\item $E^1 :=\{\alpha_1,\dots , \alpha_k,\beta_1,\dots ,\beta_l, \gamma _1,\dots ,\gamma _m,\delta_1,\dots,\delta_n\}$
(with $k+l+m+n$ distinct edges).
\item $s(e)=v$ for all $e\in E^1$, while $r(\alpha_i)=r(\beta_j)=w_1$ for all $i$, $j$, and $r(\gamma_i)= r(\delta_j)= w_2$ for all $i$, $j$.
\item $C=C_v := \{X,Y\}$ where
$$X := \{\alpha_1,\dots ,\alpha_k,\gamma_1,\dots ,\gamma _m\}, \qquad\qquad Y:= \{\beta _1,\dots, \beta_l,\delta_1,\dots,\delta_n\}.$$
\end{enumerate}
A picture of the graph $E$ for the case $k=l=m=n=2$ is shown below.

$$\xymatrixrowsep{4pc}\xymatrixcolsep{6pc}\def\labelstyle{\displaystyle}
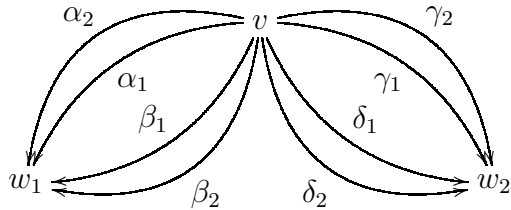
\xymatrix{
 &v \ar@/_3.5ex/[dl]^{\alpha_1} \ar@/_6ex/[dl]_{\alpha_2} \ar@/^3.5ex/[dl]_{\beta_1} \ar@/^6ex/[dl]^{\beta_2}  \ar@/^3.5ex/[dr]_{\gamma_1} \ar@/^6ex/[dr]^{\gamma_2} \ar@/_3.5ex/[dr]^{\delta_1} \ar@/_6ex/[dr]_{\delta_2}  \\
 w_1  &&w_2
 }$$
 \medskip

Observe that $H_1=\{w_1\}$ and $H_2=\{w_2\}$ are both hereditary
$C$-saturated subsets of $E^0$. However, by the next proposition,
both $H_1$ and $H_2$ generate the full algebra
$\Cstred(E,C)$.

\begin{proposition}
\label{prop:heredbreak} Let $(E,C)$ be the separated graph described
above. Then the reduced graph C*-algebra $\Cstred(E,C)$
is simple.
\end{proposition}

\begin{proof} Set $A_0 := \C v\oplus \C w_1\oplus \C w_2$, and identify $A_0$ with $\C^3$ so that $v$, $w_1$, $w_2$ correspond to $(1,0,0)$, $(0,1,0)$, $(0,0,1)$ respectively. Set $A:= C^*(E_X)$ and $B:= C^*(E_Y)$, and identify $A$ and $B$ with $M_{k+1}(\C)\times M_{m+1}(\C)$ and $M_{l+1}(\C)\times M_{n+1}(\C)$ so that $v$, $w_1$, $w_2$ correspond to
$$\bigl( \diag(1,\dots,1,0) ,\, \diag(1,\dots,1,0) \bigr), \qquad \bigl( \diag(0,\dots,0,1) ,\, 0 \bigr), \qquad \bigl( 0 ,\, \diag(0,\dots,0,1) \bigr),$$
respectively, in each case. The canonical conditional expectations $\Phi_X$ and $\Phi_Y$ correspond to the maps $\phi: M_{k+1}(\C)\times M_{m+1}(\C) \to A_0$ and $\psi: M_{l+1}(\C)\times M_{n+1}(\C) \to A_0$ given by
\begin{align*}
\phi \bigl( [a_{ij}],[a'_{ij}] \bigr) &= \frac1{k+m} \biggl( \sum_{i=1}^k a_{ii}+ \sum_{j=1}^m a'_{jj} \biggr)v + a_{k+1,k+1}w_1+ a'_{m+1,m+1}w_2  \\
\psi\bigl( [b_{ij}],[b'_{ij}] \bigr) &= \frac1{l+n} \biggl( \sum_{i=1}^l b_{ii}+ \sum_{j=1}^n b'_{jj} \biggr)v + b_{l+1,l+1}w_1+ b'_{n+1,n+1}w_2 \,.
\end{align*}

Take $P:=v$, and observe that $PA_0= \C P$ and that $P$ is a full
projection in both $A$ and $B$, so certainly $P$ will be a full
projection in $D:=A \,\bgast_{A_0}^{\text{red}}\, B=\Cstred(E,C)$.
Consider the unitaries
\begin{align*}
a &:= \bigl( \diag(v_k,0),\, \diag(v_m,0) \bigr) \in PAP  \\
b &:= \bigl( \diag(v_l,0),\, \diag(v_n,0) \bigr) \in PBP  &c &:= \bigl( \diag(u_l,0),\, \diag(u_n,0) \bigr) \in PBP.
\end{align*}
Then  \eqref{kerphipsi} holds, and moreover $\phi|_{PAP}$, $\psi|_{PBP}$ are
invariant with respect to conjugation by $a$, $b$ respectively, so
that the conditions in Proposition \ref{fullPavitzour} are
satisfied. It follows from Corollary \ref{fullPavitzourcorolar} that
$\Cstred(E,C)$ is a simple C*-algebra.
\end{proof}

%%%%%%%%%%%%%%%%%%%%%%%%%%%%%%%%%%%%%%%
\section{$K$-theory}
\label{sect:K-theory}

Our aim in this section is to compute the $K$-theory of the full
graph C*-algebras of finitely separated graphs. This will use the powerful
results in \cite{Thomsen}.

We recall here the main result from \cite{Thomsen} used in our
computations; it is a particular case of \cite[Theorem 2.7]{Thomsen}.

\begin{theorem}
\label{thm:thomsen}  Let $A_0$, $A_1$, $A_2$ be separable
C*-algebras. Assume that $i_k\colon A_0\to A_k$, for $k=1,2$, are
embeddings, and that $A_0$ is finite-dimensional. Let $j_k\colon
A_k\to A_1 \,\bgast_{A_0}\, A_2$, for $k=1,2$, be the canonical maps. Then there is a $6$-term
exact sequence:
\begin{equation}
\label{eq:6first}
\begin{CD}
K_0(A_0) @>{\;({i_1}_*,{i_2}_*)\;}>> K_0(A_1)\oplus K_0(A_2) @>{\;{j_1}_*-{j_2}_*\;}>> K_0(A_1 \,\bgast_{A_0}\, A_2)\\
@AAA & & @VVV \\
K_1(A_1 \,\bgast_{A_0}\, A_2) @<{\;{j_1}_*-{j_2}_*\;}<< K_1(A_1)\oplus K_1(A_2)
@<{\;({i_1}_*,{i_2}_*)\;}<< K_1(A_0)
\end{CD}
\end{equation}
\end{theorem}

For some direct applications of this theorem to the K-theory of C*-algebras
of separated graphs, see \cite[Section 5]{Dun}.

In order to state our result, we need some preparation. If $E$ is a
row-finite (non-separated) graph, we will denote by $A_E'$ the
adjacency matrix of $E$, that is, the matrix $(a(v,w))_{v,w\in E^0}$ in $\Z^{E^0\times E^0}$ where $a(v,w) := |s_E^{-1}(v) \cap r_E^{-1}(w)|$, that is, the number of arrows from $v$ to $w$ in $E$. Write $A^t_E$ and $1$ for
the matrices in $\Z^{E^0\times E^0\setminus\Si(E)}$ which result
from the transpose of $A'_E$ and from the identity matrix after removing the
columns indexed by sinks. Then the $K$-theory of $C^*(E)$ is
given by the formulas
\begin{align}
K_0(C^*(E)) &\cong \coker \bigl( 1-A^t_E\colon
\Z^{(E^0\setminus \Si (E))}\longrightarrow \Z^{(E^0)} \bigr)  \label{eq:K0THCstarE} \\
K_1(C^*(E)) &\cong \ker \bigl( 1-A^t_E\colon \Z^{(E^0\setminus \Si
(E))}\longrightarrow \Z^{(E^0)} \bigr)  \label{eq:K1THCstarE}
\end{align}
\cite[Theorem 3.2]{RS}. Further, the formulation of \cite[Theorem 3.2]{RS}
given in \cite[Theorem 2.3.9]{TomfMal} shows that the isomorphism
of \eqref{eq:K0THCstarE} sends $[v]$ to the coset of $\delta_v$
for all $v\in E^0$, where $(\delta_v)_{v\in E^0}$ denotes the canonical basis of $\Z^{(E^0)}$.

We now present a corresponding result for any finitely separated graph $(E,C)$. The \emph{adjacency matrix} of
$(E,C)$ is the matrix $A_{(E,C)}' :=(a(v,w))_{v,w\in E^0}$ such that the entry
$a(v,w)$ is the function $X \mapsto a_X(v,w)$ in $\Z^{C_v}$ where $a_X(v,w)$ equals the number of arrows in $X$ from
$v$ to $w$, for any $v,w\in E^0$ and $X\in C_v$.
We denote by $1_C\colon
\Z^{(C)}\rightarrow \Z^{(E^0)}$ and $A_{(E,C)}^t\colon \Z^{(C)}\rightarrow \Z^{(E^0)}$ the
homomorphisms defined by
$$1_C(\delta_X)= \delta_v \qquad\text{and}\qquad A^t_{(E,C)}(\delta_X)= \sum _{w\in
E^0}a_X(v,w)\delta_w \qquad\quad (v\in E^0,\;X\in C_v ),$$
where $(\delta_X)_{X\in C}$ denotes the canonical basis of $\Z^{(C)}$.

With this notation, the $K$-theory of $C^*(E,C)$ has formulas which
look very similar to the ones for the non-separated case:

\begin{theorem}
\label{thm:KTHsepgraph} Let $(E,C)$ be a finitely separated graph,
and adopt the notation above. Then the $K$-theory of $C^*(E,C)$ is
given as follows:
\begin{align}
K_0(C^*(E,C)) &\cong \coker \bigl( 1_C-A_{(E,C)}^t\colon \Z^{(C)}\longrightarrow \Z^{(E^0)} \bigr),  \label{eq:K0THCsep} \\
K_1(C^*(E,C)) &\cong \ker \bigl( 1_C-A_{(E,C)}^t\colon \Z^{(C)}\longrightarrow \Z^{(E^0)} \bigr).  \label{eq:K1THCsep}
\end{align}
Further:
\begin{equation}  \label{vtodeltav}
\text{The isomorphism of \eqref{eq:K0THCsep} sends\ }[v] \text{\ to the coset of\ } \delta_v \text{\ for all\ } v\in E^0.
\end{equation}
\end{theorem}

\begin{proof}
Since $K$-theory is continuous, we may reduce to the case where $E$
is a finite graph by using Proposition \ref{FSGrC*functor} and \cite[Proposition 3.5 and comments after
Definition 8.4]{AG}. Set $A_0 := C^*((E^0,\emptyset),\emptyset) =C_0(E^0)$, which is a finite dimensional commutative C*-algebra under our current assumption. There is an isomorphism $\kappa: K_0(A_0) \to \Z^{(E^0)}$ sending $[v] \mapsto \delta_v$ for $v\in E^0$.

For a finite separated graph $(E,C)$ (meaning that $E^0$, $E^1$, and
$C$ are all finite), we will show the results by induction on $|C|$.
The case where $|C|\le1$ follows from the results for non-separated
graphs. Assume that $n>1$ and that the results are true for finite
separated graphs $(E',C')$ with $|C'|<n$. Let $(E,C)$ be a finite
separated graph with $|C|=n$, and select $X\in C_v$ for some $v\in
E^0 \setminus \Si(E)$. Let $C' :=C\setminus \{X\}$, and consider the
separated graphs $(E_1,C')$ and $(E_2,\{X\})$, where
$(E_1)^0=(E_2)^0=E^0$ and $(E_1)^1=\bigsqcup _{X\in C'} X$,
$(E_2)^1= X$. Then we have
$$C^*(E,C)=C^*(E_1,C') \,\bgast_{A_0}\, C^*(E_2,\{X\}),$$
relative to the canonical embeddings $i_1: A_0 \to C^*(E_1,C')$ and
$i_2: A_0 \to C^*(E_2,\{X\})$ corresponding to the inclusion
morphisms $((E^0,\emptyset),\emptyset) \to (E_1,C')$ and
$((E^0,\emptyset),\emptyset) \to (E_2,\{X\})$ in $\FSGr$. Therefore
we can apply Thomsen's result and the induction hypothesis to
compute $K_0(C^*(E,C))$. By induction, there is a commutative
diagram as follows, where $\pi_{C'}$ and $\pi_{\{X\}}$ are the
obvious quotient maps.
$$\begin{CD}
K_0(C^*(E_1,C')) @<{\;{i_1}_*\;}<<  K_0(A_0) @>{\;{i_2}_*\;}>>  K_0(C^*(E_2,\{X\}))  \\
@V{\cong}VV  @V{\kappa}V{\cong}V  @VV{\cong}V  \\
\coker(1_{C'}-A^t_{(E_1,C')}) @<{\;\pi_{C'}\;}<<  \Z^{(E^0)}
@>{\;\pi_{\{X\}}\;}>>  \coker(1_{\{X\}}-A^t_{(E_2,\{X \})})
\end{CD}$$
(The diagram is commutative because \eqref{vtodeltav} holds for the
cases $(E_1,C')$ and $(E_2,\{X\})$.) Since $K_1(A_0)=0$, it follows
from Theorem \ref{thm:thomsen} that $K_0(C^*(E,C))$ is isomorphic to
the cokernel of the map
\begin{equation}
\label{eq:K0proof}
\begin{CD}
\Z^{(E^0)} @>{\;(\pi_{C'},\pi_{\{X\}})\;}>> \Z^{(E^0)}/
(1_{C'}-A^t_{(E_1,C')})\Z^{(C')}\bigoplus \Z^{(E^0)}/
(1_{\{X\}}-A^t_{(E_2,\{X \})})\Z^{(\{X\})} \,,
\end{CD}
\end{equation}
via an isomorphism that sends $[v]$ to the coset of $(\delta_v+
(1_{C'}-A^t_{(E_1,C')})\Z^{(C')},\, 0)$ for $v\in E^0$. The cokernel
of \eqref{eq:K0proof} is easily seen to be isomorphic to
$$\Z^{(E^0)} {\biggm/} \biggl( (1_{C'}-A^t_{(E_1,C')}) \Z^{(C')} + (1_{\{X\}}-A^t_{(E_2,\{X \})})
\Z^{(\{X\})} \biggr) = \Z^{(E^0)} {\biggm/} \biggl(
(1_{C}-A^t_{(E,C)}) \Z^{(C)} \biggr),$$ in view of the exact
sequence
\begin{multline*}  \begin{CD}
\Z^{(E^0)} @>{\;(\pi_{C'},\pi_{\{X\}})\;}>>
\coker(1_{C'}-A^t_{(E_1,C')}) \oplus \coker(1_{\{X\}}-A^t_{(E_2,\{X
\})})
\end{CD}  \\
\begin{CD}
@>{\;(q_1,-q_2)\;}>>\coker(1_C-A_{(E,C)}^t) @>>> 0 \,,
\end{CD} \end{multline*}
 where $q_1$ and $q_2$ are the natural quotient maps.
We thus obtain both \eqref{eq:K0THCsep} and \eqref{vtodeltav}.

Now we want to compute $K_1(C^*(E,C))$. From (\ref{eq:6first}) and the above observations, we
get a short exact sequence:
\begin{equation}
\label{eq:K1comp1} 0 \to K_1(C^*(E_1,C'))\oplus K_1(C^*(E_2,\{X\}))
\to K_1(C^*(E,C)) \to \ker (\pi_{C'},\pi_{\{X\}}) \to 0.
\end{equation}
Set $A :=1_{C'}-A^t_{(E_1,C')}$ and $B
:=1_{\{X\}}-A^t_{(E_2,\{X\})}$.

We distinguish two cases.

\noindent {\bf Case 1:} $X$ consists of a single loop at $v$. In this case, $B=0$, and so
$\pi_{\{X\}}$ is injective and $\ker (\pi_{C'},\pi_{\{X\}})=0$. By
using the induction hypothesis for $K_1$, we get
$$K_1(C^*(E,C))\cong \ker (A)\oplus \Z\delta_X =\ker (1_C-A^t_{(E,C)}),$$
as desired.

\noindent {\bf Case 2:} $|X|>1$ or $X$ consists of a single edge
from $v$ to some different vertex. Then $B\colon \Z\delta_X \to
\Z^{(E^0)}$ is injective, so that $K_1(C^*(E_2,\{X\}))=\ker (B)=0$,
and we get from (\ref{eq:K1comp1}):
\begin{equation}  \label{K1EC}
\begin{aligned}
K_1(C^*(E,C)) &\cong K_1(C^*(E_1,C')) \oplus \ker(\pi_{C'},\pi_{\{X\}}) \\
 &\cong \ker \bigl( A\colon \Z^{(C')}\to \Z^{(E^0)} \bigr) \oplus \bigl( A(\Z^{(C')})\cap B(\Z\delta_X) \bigr).
\end{aligned}
\end{equation}
Now using that $A(\Z^{(C')})\cap B(\Z\delta_X)$ is cyclic and $B$ is
injective, it is straightforward to show that the last direct sum in \eqref{K1EC} is isomorphic to
$$\ker \bigl( \begin{pmatrix} A & B
\end{pmatrix} \colon \Z^{(C)}\to \Z^{(E^0)} \bigr).$$
Since $\begin{pmatrix} A & B
\end{pmatrix} =1_C-A^t_{(E,C)}$, we get the desired result for $K_1(C^*(E,C))$.
\end{proof}

As an example, we consider the separated graph
$(E(m,n),C(m,n))$ of Example
\ref{ex:EmnCmn}, for $1\le m\le n$. Now \eqref{isomatUnc} and Theorem \ref{thm:KTHsepgraph} give
\begin{align*}
K_0(U^{\text{nc}}_{m,n}) &\cong \coker \left( \begin{pmatrix}
1 & 1 \\ -n & -m
\end{pmatrix} : \Z^2\to \Z^2 \right) \cong \begin{cases}  \Z & \text{if } n=m\\ \Z_{n-m} & \text{if } n>m \end{cases} \\
K_1(U^{\text{nc}}_{m,n}) &\cong \ker\left( \begin{pmatrix}
1 & 1 \\ -n & -m
\end{pmatrix} : \Z^2\to \Z^2 \right) \cong \begin{cases} \Z & \text{if } n=m\\
0 & \text{if } n>m
\end{cases} .
\end{align*}
This confirms a conjecture of McClanahan \cite[Conjecture, p.~1067]{McCla2}, and recovers
\cite[Corollary 2.4]{McCla1} in the case $n=m$.

%%%%%%%%%%%%%%%%%%%%%%%%%%%%%%%%%%%%%%%
\section{Relationships with McClanahan's examples}
\label{sect:relMcClan}

We show here that  the reduced C*-algebra of the separated graph
$(E(m,n),C(m,n))$ of Example \ref{ex:EmnCmn}
is Morita-equivalent to the C*-algebra constructed by McClanahan
in \cite[Example 3.12]{McCla3}. Let us recall the definition in \cite{McCla3}. Let
$$(\mathcal B,\Psi) :=(M_{n+m}(\C),\Psi_1) \,\bgast_{\C^2}\,(M_2(\C),\Psi_2)$$
be the reduced amalgamated product over $\C^2$ of the algebras
$M_{n+m}(\C)$ and $M_2(\C)$, with respect to the conditional
expectations defined by
\begin{gather*}
\Psi _1 \bigl( (a_{ij}) \bigr) = \biggl( \frac{1}{n}\sum_{i=1}^n a_{ii}\,,\; \frac{1}{m}\sum _{j=1}^m a_{n+j,n+j} \biggr),  \\
\Psi_2 \biggl( \begin{pmatrix}
a & b \\ c & d
\end{pmatrix} \biggr) = (a, d).
\end{gather*}

\begin{proposition}  \label{isoMcClan1}
Let $1<m\le n$, let $(E,C) := (E(m,n),C(m,n))$ be the separated graph described in Example {\rm\ref{ex:EmnCmn}}, and let
$$(\calA, \Phi) := \Cstred(E,C) \equiv (M_{n+1}(\C),\Phi _1) \,\bgast_{\C^2}\, (M_{m+1}(\C),\Phi _2)$$ be the corresponding reduced C*-algebra.
Let $T :=v\calA v$ be the corner of $\calA$ corresponding to
$v\in E^0$, and observe that $\Phi$ restricts to a faithful, completely positive conditional
expectation $\phi : T \to \C{\cdot} 1_T$. Then we have a $*$-isomorphism
$$(\mathcal B,\Psi )\cong (M_2(\C)\otimes T, \Psi_2\otimes \phi).$$
\end{proposition}

\begin{proof}
We are going to use again the characterization of the reduced
amalgamated product. Let $e_{ij}$ and
$f_{ij}$ denote the canonical matrix units in $M_{n+m}(\C)$ and $M_2(\C)$, respectively. There exist unital *-homomorphisms $\sigma_1\colon M_{n+m}(\C)\to
M_2(\C)\otimes T$ and $\sigma _2\colon M_2(\C)\to
M_2(\C)\otimes T$ such that
$$\sigma_1(e_{ij})= \begin{cases}  f_{11}\otimes \alpha_i\alpha^*_j  &(1\le i,j\le n)\\
f_{12}\otimes \alpha_i\beta^*_{j-n}  &(1\le i\le n<j\le n+m)\\  f_{21}\otimes \beta_{i-n}\alpha^*_j  &(1\le j\le n<i\le n+m)\\  f_{22}\otimes \beta_{i-n}\beta^*_{j-n}  &(n<i,j\le n+m)
\end{cases}$$
and $\sigma _2(f_{ij}) = f_{ij}\otimes 1_T$ for $i,j=1,2$.
Now all the conditions in the definition
of the reduced amalgamated product are easily verified, with the
exception of (4), that needs some work.

Observe that $\ker \Psi_2$ is spanned by $\{f_{12},f_{21}\}$, while $\ker \Psi_1$ is spanned by the set
$$\Xi :=\{ e_{kl} : 1\le k,l\le n+m,\;  k\ne l\} \cup \{\epsilon _i : 1\le i\le n\} \cup \{ \ol{\epsilon}_j : 1\le j\le m\},$$
with $\epsilon _i :=e_{ii}-\frac{1}{n}\sum _{t=1}^n e_{tt}$ and
$\ol{\epsilon}_j :=e_{j+n,j+n}-\frac{1}{m}\sum _{s=1}^m
e_{s+n,s+n}$. We note that $\sigma_1(\epsilon_k)= f_{11}\otimes \lambda(\alpha_k)$ and $\sigma_1(\ol{\epsilon}_k)= f_{22}\otimes \lambda(\beta_k)$, where $\lambda (\alpha _k) := \alpha_k \alpha _k^*-\frac{1}{n}v$ and $\lambda (\beta _k) := \beta_k \beta _k^*-\frac{1}{m}v$.

For subsets $T_1$, $T_2$ of an algebra $\mathcal H$, denote by
$\Lambda ^{{\rm o}}(T_1,T_2)$ the set of all elements of $\mathcal
H$ of the form $a_1a_2 \cdots a_r$, where $a_j\in T_{i_j}$ and
$i_1\ne i_2$, $i_2\ne i_3$,\dots , $i_{r-1}\ne i_r$. With this
notation, to verify (4) it will be enough to show that $(\Psi
_2\otimes \phi )(a_1a_2\cdots a_r)=0$ for all  $a_1a_2\cdots a_r$ in
$\Lambda ^{{\rm o}}(\{ \sigma_2(f_{12}), \sigma_2(f_{21})
\},\sigma_1(\Xi) )$. This is, of course, clear for $r=1$. We claim
that:
\begin{enumerate}
\item[(I)] For $r\ge2$, any word $a_1a_2\cdots a_r \in \Lambda ^{{\rm o}}(\{
\sigma_2(f_{12}), \sigma_2(f_{21}) \},\sigma_1(\Xi) )$ is either
zero or has the form $f_{ij}\otimes d$ with $d\in \Lambda^{{\rm
0}}(T_\alpha,T_\beta)$, where
\begin{align*}
T_\alpha &:= \{ \alpha_k,\, \alpha^*_k,\, \lambda(\alpha_k) \mid 1\le k\le n\} \cup \{\alpha_k\alpha^*_l \mid 1\le k,l\le n,\; k\ne l \}  \\
T_\beta &:= \{ \beta_k,\, \beta^*_k,\, \lambda(\beta_k) \mid 1\le k\le m\} \cup \{ \beta_k\beta^*_l \mid 1\le k,l\le m,\; k\ne l\} \,,
\end{align*}
and also
\begin{enumerate}
\item If $a_r\in \sigma_1(\Xi)$ and $j=1$, then $d$ ends in one of $\alpha^*_l$ or $\alpha_k \alpha^*_l$ (with $k\ne l$) or $\lambda(\alpha_k)$;
\item If $a_r\in \sigma_1(\Xi)$ and $j=2$, then $d$ ends in one of $\beta^*_l$ or $\beta_k \beta^*_l$ (with $k\ne l$) or $\lambda(\beta_k)$.
\end{enumerate}
\end{enumerate}

For $c\in \sigma_1(\Xi)$, observe that
\begin{enumerate}
\item $\sigma_2(f_{12}) c\ne 0$ if and only if $c=  f_{21}\otimes \beta_{k-n}\alpha^*_l$; or $c= f_{22}\otimes \beta_{k-n}\beta^*_{l-n}$ with $k\ne l$; or $c= f_{22}\otimes \lambda(\beta_k)$.
\item $\sigma_2(f_{21}) c\ne 0$ if and only if $c= f_{11}\otimes \alpha_k\alpha^*_l$ with $k\ne l$; or $c= f_{11}\otimes \lambda(\alpha_k)$; or $c=  f_{12}\otimes \alpha_k\beta^*_{l-n}$.
\end{enumerate}
With the aid of these observations, the claim is easily established by induction on $r$.

Since $(\calA,\Phi)$ is a reduced amalgamated product,
$\Lambda^{{\rm o}}(T_\alpha,T_\beta) \subseteq \ker \Phi$. Hence, we
conclude that $\Lambda ^{{\rm o}}(\{ \sigma_2(f_{12}),
\sigma_2(f_{21}) \},\sigma_1(\Xi) ) \subseteq
\ker(\Psi_2\otimes\phi)$, as desired.
\end{proof}

The case $n=m$ is special. In this case, the reduced graph
C*-algebra admits a faithful trace and has minimal projections,
as we will see. We establish this by finding a Morita equivalence with a different example of McClanahan's, namely \cite[Example 4.1]{McCla3}. This example is given as follows:
$$(\calC,\Psi ) := (M_n(\C),\text{tr}_n) \,\bgast_{\C}\, (C(\mathbb T), \tau) \, ,$$
where $\text{tr}_n$ is the normalized matrix trace on $M_n(\C)$ and $\tau$ is the usual faithful trace on $C(\mathbb T)$. Then we
have:

\begin{proposition}  \label{isoMcClan2}
Let $n>1$, let $(E,C) := (E(n,n),C(n,n))$ be the separated graph described in Example {\rm\ref{ex:EmnCmn}}, and let
$$(\calA, \Phi) := \Cstred(E,C) \equiv (M_{n+1}(\C),\Phi _1) \,\bgast_{\C^2}\, (M_{n+1}(\C),\Phi _2)$$ be the corresponding reduced C*-algebra.
Let $T :=v\calA v$ be the corner of $\calA$ corresponding to $v\in E^0$, and observe that $\Phi$ restricts to a faithful, completely positive conditional
expectation $\phi : T \to \C{\cdot} 1_T$. There is a
$*$-isomorphism
$$(\calC,\Psi )\cong (T, \phi).$$
\end{proposition}

\begin{proof}
We are going to use again the characterization of the reduced
amalgamated product. Let $u$ denote the standard unitary generator of $C(\mathbb T)$, and let $e_{ij}$ be the canonical matrix units in $M_n(\C)$. There exist unital *-homomorphisms $\sigma_1: M_n(\C) \rightarrow T$ and $\sigma_2: C(\mathbb T) \rightarrow T$ such that
$\sigma_1(e_{ij})= \alpha_i\alpha^*_j$ for all $i$, $j$, and
$$\sigma_2(u)= U := \sum_{j=1}^n \beta_j\alpha^*_j \,.$$
Conditions (1), (2), (5) in the definition of the reduced amalgamated product are easily verified, as is the first part of (3), namely, that $\phi\circ \sigma_1= \text{tr}_n$.

As in the proof of Proposition \ref{isoMcClan1}, define
\begin{align*}
T_\alpha &:= \{ \alpha_k,\, \alpha^*_k,\, \lambda(\alpha_k) \mid 1\le k\le n\} \cup \{\alpha_k\alpha^*_l \mid 1\le k,l\le n,\; k\ne l \}  \\
T_\beta &:= \{ \beta_k,\, \beta^*_k,\, \lambda(\beta_k) \mid 1\le k\le n\} \cup \{ \beta_k\beta^*_l \mid 1\le k,l\le n,\; k\ne l\} \,,
\end{align*}
where  $\lambda (\alpha _k) := \alpha_k \alpha _k^*-\frac{1}{n}v$
and $\lambda (\beta _k) := \beta_k \beta _k^*-\frac{1}{n}v$, and
observe that $\Lambda^{{\rm o}}(T_\alpha,T_\beta) \subseteq \ker
\Phi$. Since any nonzero power of $U$ is a linear combination of
elements of $\Lambda^{{\rm o}}(T_\alpha,T_\beta)$, we see that
$\phi\sigma_2(u^t)= \Phi(U^t)=0$ for all $t\ne 0$. It follows that
$\phi\circ \sigma_2 = \tau$, verifying (3).

Observe that $\ker \text{tr}_n$ and $\ker \tau$ are the closed linear spans of the sets
\begin{align*}
\Xi &:=\{ e_{kl} : 1\le k,l\le n,\;  k\ne l\} \cup \{\epsilon _i : 1\le i\le n\}  \\
\Upsilon &:= \{ u^t : t\in \Z \setminus \{0\} \,\},
\end{align*}
with $\epsilon _i :=e_{ii}-\frac{1}{n}\sum _{t=1}^n e_{tt}$ as
before. Also, $\sigma_1(\epsilon_k)= \lambda(\alpha_k)$. To verify
condition (4), it is enough to show that $\phi(a_1a_2\cdots a_r)=0$
for all $a_1a_2\cdots a_r$ in $\Lambda^{{\rm
0}}(\sigma_1(\Xi),\sigma_2(\Upsilon))$. This is clear for $r=1$.

We claim that
\begin{enumerate}
\item[(I)] Each element $a_1a_2\cdots a_r \in \Lambda^{{\rm o}}(\sigma_1(\Xi),\sigma_2(\Upsilon))$ with $r\ge 2$ and $a_r\in \sigma_2(\Upsilon)$ can be written as a linear combination of terms $wU^t$ such that $w\in \Lambda^{{\rm o}}(T_\alpha,T_\beta)$ and one of the following holds:
\begin{enumerate}
\item $t\ge1$ and $w$ ends in one of $\alpha^*_l$ or $\alpha_k\alpha^*_l$ (with $k\ne l$) or $\lambda(\alpha_k)$;
\item $t\le 0$ and $w$ ends in one of $\beta^*_l$ or $\beta_k\beta^*_l$ (with $k\ne l$) or $\lambda(\beta_k)$.
\end{enumerate}
\end{enumerate}
The cases when $r=2$ and $a_2$ is a positive power of $U$ are clear. The other $r=2$ cases follow from the facts that
\begin{align*}
\sigma_1(e_{kl})U^{-1} &= \alpha_k\beta^*_l  &\sigma_1(\epsilon_k)U^{-1} &= \alpha_k\beta^*_k- \tfrac1nU^{-1}
\end{align*}
for all $k$, $l$. The remainder of claim (I) is proved by induction on $r$, with the help of the following observations:
\begin{align*}
U\sigma_1(e_{kl})U &= \beta_k\alpha^*_lU  &&&U\sigma_1(e_{kl})U^{-1} &= \beta_k\beta^*_l  \\
U^{-1}\sigma_1(e_{kl})U &= U^{-1}\alpha_k\alpha^*_lU  &&&U^{-1}\sigma_1(e_{kl})U^{-1} &= U^{-1}\alpha_k\beta^*_l  \\
U\sigma_1(\epsilon_k)U &= \beta_k\alpha^*_kU - \tfrac1nU^2  &&&U\sigma_1(\epsilon_k)U^{-1} &= \lambda(\beta_k)  \\
U^{-1}\sigma_1(\epsilon_k)U &= U^{-1}\lambda(\alpha_k)U  &&&U^{-1}\sigma_1(\epsilon_k)U^{-1} &= U^{-1}\alpha_k\beta^*_k- \tfrac1nU^{-2}
\end{align*}
for all $k$, $l$.

Finally, we claim that
\begin{enumerate}
\item[(II)] Every element $a_1a_2\cdots a_r \in \Lambda^{{\rm o}}(\sigma_1(\Xi),\sigma_2(\Upsilon))$ can be written as a linear combination of elements of $\Lambda^{{\rm o}}(T_\alpha,T_\beta)$.
\end{enumerate}
This is clear when $r=1$, and it follows directly from (I) when $r\ge 2$ and $a_r\in \sigma_2(\Upsilon)$, just by expanding the factors $U^t$. When $r\ge2$ and $a_r\in \sigma_1(\Xi)$, we obtain (II) from (I) with the help of the facts that
\begin{align*}
U\sigma_1(e_{kl}) &= \beta_k\alpha^*_l  &U\sigma_1(\epsilon_k) &= \beta_k\alpha^*_k- \tfrac1nU
\end{align*}
for all $k$, $l$. Since $\Lambda^{{\rm o}}(T_\alpha,T_\beta)
\subseteq \ker \Phi$, we conclude from (II) that $\Lambda^{{\rm
0}}(\sigma_1(\Xi),\sigma_2(\Upsilon)) \subseteq \ker \phi$, as
desired.
\end{proof}

\begin{corollary}  \label{cor:C*Enn}
Let $n>1$ and $(\calA, \Phi) := \Cstred(E,C)$ as in Proposition {\rm\ref{isoMcClan2}}. Then $\calA$ is a simple C*-algebra with a faithful trace. It has stable rank $1$, but does not have real rank zero.
\end{corollary}

\begin{proof} By Proposition {\rm\ref{isoMcClan2}}, $\calA$ is Morita equivalent to McClanahan's example $\calC$, where $(\calC,\Psi ) := (M_n(\C),\text{tr}_n) \,\bgast_{\C}\, (C(\mathbb T), \tau)$. We show that $\calC$ has the described properties. Simplicity follows from either \cite[Proposition 3.3]{McCla3} or Proposition \ref{prop:simple2}, and $\Psi$ is a trace because $\text{tr}_n$ and $\tau$ are traces (\cite[Proposition 1.4]{avitzour} or \cite[2.5.3]{VDN}). It is faithful because $\text{tr}_n$ and $\tau$ are faithful.

Next, since $C(\mathbb T)$ is a diffuse abelian algebra with respect to $\tau$, meaning that $\tau$ is
given by an atomless measure on $\mathbb T$, it follows from \cite[Proposition 3.4]{dykema2}
that $\calC$ has stable rank $1$.

Finally, we consider the K-theory of the full free product algebra
$\calC_{\text{full}} := M_n(\C) * C(\mathbb T)$. Write $e_{ij}$ for
the canonical matrix units in the copies of $M_n(\C)$ appearing in
the different algebras under consideration.  Since $K_0(M_n(\C))$
and $K_0(C(\mathbb T))$ are infinite cyclic, with generators
$[e_{11}]$ and $[1]$, respectively, it follows from Theorem
\ref{thm:thomsen} that $K_0(\calC_{\text{full}})$ is infinite
cyclic, with generator $[e_{11}]$. McClanahan showed in
\cite[Corollary 8.7]{McCla4} (cf.~\cite[Example 4.1]{McCla3}) that
the natural map $\calC_{\text{full}} \to \calC$ induces isomorphisms
in K-theory. (This also follows from \cite[Theorem 4.1]{Germain}.)
Consequently, $K_0(\calC)$ is infinite cyclic, with generator
$[e_{11}]$. The faithful trace $\Psi$ on $\calC$ thus takes values
in $(1/n)\Z$, and it follows that $e_{11}$ is a minimal projection
in $\calC$. Therefore $\calC$ cannot have real rank $0$.
\end{proof}

%%%%%%%%%%%%%%%%%%%%%%%%%%%%%%%%%%%%%%
\section{Problems}
\label{sect:openproblems}

In this final section, we discuss some open problems which arise
naturally in this investigation.

\begin{openproblem}
\label{opideals} Compute the lattices of closed ideals of the full and reduced
graph C*-algebras of a finitely separated graph $(E,C)$, in terms of graph-theoretic data. In
particular, find characterizations of simplicity of $C^*(E,C)$ and/or $\Cstred(E,C)$ in terms of $(E,C)$.

For the ordinary graph C*-algebra $C^*(E)$ of a countable graph $E$, the lattice of gauge-invariant closed ideals was characterized in \cite[Theorem 3.6, Corollaries 3.8, 3.10]{BHRS}; this gives the full lattice of closed ideals in case $E$ satisfies Condition (K) ([ibid], \cite[Theorem 3.5]{DrTo}). A characterization of simplicity of $C^*(E)$ was found earlier, in \cite[Theorem 12]{Szy}. For the Leavitt path algebra of a separated graph $(E,C)$, the lattice of trace ideals was characterized in \cite[Theorem 6.11]{AG}, and necessary and sufficient conditions for ``trace-simplicity'' of $L(E,C)$ were obtained in \cite[Theorem 7.1]{AG}.
\end{openproblem}

\begin{openproblem}
\label{opequals} Find conditions when the full graph C*-algebra
of a finitely separated graph $(E,C)$ equals the reduced one. Observe that
this is always the case for non-separated graphs (Theorem
\ref{thm:red-prod-inj}(2)). A necessary condition for equality is
that the full C*-algebra needs to be exact (Remark
\ref{rm:exact}).

Exactness often fails, however, as shown by Duncan \cite{Dun}. First, if there is a vertex $v\in E^0$ at which there are two loops lying in different members of $C_v$, then $C^*(E,C)$ is not exact \cite[Proposition 6]{Dun}. Second, if there are vertices $v,w\in E^0$ and three edges from $v$ to $w$ which lie in distinct members of $C_v$, then  $C^*(E,C)$ is not exact \cite[Proposition 7]{Dun}.
\end{openproblem}

\begin{openproblem}
\label{pur-inf-simple} When the reduced graph C*-algebra of a finitely separated graph is
simple and infinite (as in some of the examples in Section
\ref{sect:idealsinCred}), is it purely infinite? Does it at least
have real rank zero? Both answers are positive in the non-separated case \cite[Theorem 18]{Szy}.
\end{openproblem}

\begin{openproblem}
\label{sr1rr0} When the reduced graph C*-algebra of a finitely separated graph is
simple and finite, must have it stable rank one? The answer is positive in the non-separated case \cite[Theorem 18]{Szy}. The corresponding question for real rank zero is answered negatively by Corollary \ref{cor:C*Enn}.
\end{openproblem}

\begin{openproblem}
\label{kktheory}
For free products of two \emph{nuclear} C*-algebras with faithful
states, Germain proved in \cite{Germain} that the natural map from
the full free product to the reduced one is a $KK$-equivalence and
so it induces an isomorphism in K-theory. Note that this applies to
the examples in Proposition \ref{prop:twoLeavialgs}.

 Is there a corresponding result for amalgamated
free products? This would apply in particular to all the examples of
Section \ref{sect:idealsinCred}, for which we could then compute the
$K$-theory of the reduced graph C*-algebras (thanks to Theorem
\ref{thm:KTHsepgraph}).
\end{openproblem}

\begin{openproblem}
\label{op-vmonoid}
Let $(E,C)$ be a finitely separated graph. Let $M(E,C)$ be the
abelian monoid with generators $\{a_v\mid v\in E^0\}$ and relations
given by $a_v=\sum _{e\in X} a_{r(e)}$ for all $v\in E^0$
and all $X\in C_v$. It was shown in \cite[Theorem 4.3]{AG} that there
is a natural isomorphism $M(E,C)\to \mon{L(E,C)}$, sending $a_v$ to
$[v]\in \mon{L(E,C)}$, where $\mon{L(E,C)}$ is the abelian monoid of
Murray-von Neumann equivalence classes of projections in matrices
over $L(E,C)$.

 Is the natural map $M(E,C)\to \mon{C^*(E,C)}$
also an isomorphism? Equivalently, is the natural induced map
$\mon{L(E,C)}\to \mon{C^*(E,C)}$ an isomorphism?

We conjecture that the answer to this question is
positive. This is certainly the case for non-separated graphs
(\cite[Theorem 7.1]{AMP}). If the answer is positive, it would follow, as in \cite[Corollary 4.5]{AG}, that every conical abelian monoid is isomorphic to $\mon{C^*(E,C)}$ for some finitely separated graph $(E,C)$.
\end{openproblem}


\begin{thebibliography}{12}

\bibitem{AG} P. Ara, K. R. Goodearl, \emph{Leavitt path algebras of
separated graphs}, to appear in J. reine angew. Math.;
arXiv:1004.4979v2 [math.RA].


\bibitem{AMP} P. Ara, M. A. Moreno, E. Pardo, \emph{Nonstable $K$-theory for graph
algebras}, Algebr. Represent. Theory \textbf{10} (2007), 157--178.


\bibitem{avitzour} D. Avitzour, \emph{Free products of C*-algebras},
Trans. Amer. Math. Soc. \textbf{271} (1982), 423--435.

\bibitem{BHRS} T. Bates, J. H. Hong, I. Raeburn, and W. Szyma\'nski, \emph{The ideal structure of the C*-algebras of infinite graphs},
Illinois J. Math. \textbf{46} (2002), 1159--1176.

\bibitem{BP} T. Bates and D. Pask, \emph{C*-algebras of labelled graphs}, J. Operator Theory \textbf{57} (2007), 207--226.

\bibitem{Brown} L. G. Brown, \emph{Ext of certain free product $C^{\ast}
$-algebras}, J. Operator Theory \textbf{6} (1981), 135--141.

\bibitem{BO} N. P. Brown, N. Ozawa, C*-Algebras and
Finite-Dimensional Approximations, Graduate Studies in Math.
88, American Math. Soc., Providence, RI, 2008.

\bibitem{Cuntz77}  J. Cuntz, \emph{Simple $C\sp*$-algebras generated by
isometries}, Comm. Math. Phys.  {\bf 57}  (1977), 173--185.

\bibitem{Cuntz81}  \bysame, \emph{$K$-theory for certain $C^{\ast} $-algebras},
Ann. of Math. {\bf 113}  (1981), 181--197.

\bibitem{CuntzKrieger}  J. Cuntz, W. Krieger, \emph{A class of $C^{\ast} $-algebras
and topological Markov chains},  Invent. Math.  {\bf 56}  (1980),
251--268.

\bibitem{DrTo} D. Drinen and M. Tomforde, \emph{The C*-algebras of arbitrary graphs}, Rocky Mtn. J. Math. \textbf{35} (2005), 105--135.

\bibitem{Dun} B. L. Duncan, \emph{Certain free products of graph operator algebras}, J. Math. Anal. Applic. \textbf{364} (2010), 534--543.

\bibitem{dykema} K. J. Dykema, \emph{Faithfulness of free product states},
J. Funct. Anal. \textbf{154} (1998), 323--329.

\bibitem{dykema2} \bysame, \emph{Simplicity and the stable rank of some free product C*-algebras}, Trans. Amer. Math. Soc. \textbf{351} (1999), 1--40.

\bibitem{Germain} E. Germain, \emph{$KK$-theory of reduced free-product C*-algebras}
Duke Math. J. 82 (1996), 707--723.


\bibitem{ivanov} N. A. Ivanov, \emph{On the structure of some reduced amalgamated
free product C*-algebras}, Internat. J. Math. \textbf{22} (2011),
281--306.

\bibitem{KPR} A. Kumjian, D. Pask, I. Raeburn, \emph{Cuntz-Krieger algebras
of directed graphs}, Pacific J. Math. \textbf{184} (1998), 161--174.

\bibitem{KPRR} A. Kumjian, D. Pask, I. Raeburn, J. Renault, \emph{Graphs,
groupoids, and Cuntz-Krieger algebras}, J. Funct. Anal. \textbf{144}
(1997), 505--541.

\bibitem{Lance} E. C. Lance, Hilbert C*-Modules. A Toolkit for Operator
Algebraists, London Math. Soc. Lecture Notes 210,
Cambridge Univ. Press, Cambridge, 1995.

\bibitem{McCla1} K. McClanahan, \emph{$C\sp *$-algebras generated by elements of a
unitary matrix}, J. Funct. Anal. \textbf{107} (1992), 439--457.

\bibitem{McCla2} \bysame, \emph{$K$-theory and ${\rm
Ext}$-theory for rectangular unitary $C\sp *$-algebras}, Rocky
Mountain J. Math. \textbf{23} (1993), 1063--1080.

\bibitem{McCla3} \bysame, \emph{Simplicity of reduced amalgamated products of
C*-algebras} Canad. J. Math. \textbf{46} (1994), 793--807.

\bibitem{McCla4} \bysame, \emph{K-theory for certain reduced free products of C*-algebras}
J. Operator Theory, \textbf{33} (1995), 201--221.


\bibitem{Raeburn} I. Raeburn, Graph Algebras, CBMS Regional Conf. Series in
Math. 103, American Math.
Soc., Providence, RI, 2005.

\bibitem{RS} I. Raeburn and W. Szyma\'nski , \emph{Cuntz-Krieger algebras of infinite graphs and matrices}, Trans. Amer. Math. Soc. \textbf{356} (2004), 39--59.

\bibitem{Renault} J. Renault, A Groupoid Approach to $C^{\ast} $-Algebras,
Lecture Notes in Math. 793, Springer-Verlag, Berlin, 1980.

\bibitem{Szy} W. Szyma\'nski, \emph{Simplicity of Cuntz-Krieger algebras of infinite matrices}, Pacific J. Math. \textbf{199} (2001), 249--256.

\bibitem{Thomsen} K. Thomsen, \emph{On the $KK$-theory and the $E$-theory
of amalgamated free products of C*-algebras}, J. Func.
Anal. \textbf{201} (2003), 30--56.

\bibitem{Tomf07} M. Tomforde, \emph{Uniqueness theorems and ideal structure for Leavitt
path algebras}, J. Algebra \textbf{318} (2007), 270--299.

\bibitem{TomfMal} \bysame, \emph{Structure of graph C*-algebras and generalizations}, in Graph Algebras: Bridging the Gap Between Analysis and Algebra (G. Aranda Pino, F. Perera Dom\`enech, and M. Siles Molina, Eds.),  23--83,  Universidad de M\'alaga, 2006.

\bibitem{voicu}
D. Voiculescu, \emph{Symmetries of some reduced free product
$C^\ast$-algebras}, in Operator Algebras and Their Connections with
Topology and Ergodic Theory (Busteni, 1983), 556--588, Lecture Notes
in Math. 1132, Springer-Verlag, Berlin, 1985.

\bibitem{VDN}
D. Voiculescu, K. J. Dykema, and A. Nica, Free Random Variables, CRM Monograph Series 1, American Math. Soc., Providence, 1992.


\end{thebibliography}
\end{document}